\documentclass[12pt]{article} % for usual tex environment
\usepackage{amsmath, graphicx}
\usepackage{amstext,amsfonts,amsbsy,eucal,amssymb, amsthm,color}

\usepackage{pgf,tikz,pgfplots}
\usetikzlibrary{arrows}

\definecolor{light-gray}{gray}{0.80}
\definecolor{light-gray2}{gray}{0.70}

\advance \topmargin by -\headheight
\advance \topmargin by -\headsep
     
\evensidemargin \oddsidemargin
\newtheorem{theorem}{Theorem}[section]
\newtheorem{proposition}[theorem]{Proposition}
\newtheorem{corollary}[theorem]{Corollary}

\theoremstyle{definition}
\newtheorem{definition}[theorem]{Definition}
\newtheorem{example}[theorem]{Example}
\newtheorem{remark}[theorem]{Remark}

\def\disp{\displaystyle}
\def\Z{\operatorname{\mathbb{Z}}}

\def\C{\operatorname{\mathbb{C}}}
\def\P{\operatorname{\mathbb{P}}}
\def\div{\operatorname{{\rm Div}}}
\def\pic{\operatorname{{\rm Pic}}}
\def\cX{\operatorname{\mathcal{X}}}
\def\cY{\operatorname{\mathcal{Y}}}
\def\cZ{\operatorname{\mathcal{Z}}}
\def\cE{\operatorname{\mathcal{E}}}

\def\ve{\varepsilon}

\begin{document}
\begin{center}
{\large {\bf Space of initial conditions and geometry of two 4-dimensional discrete Painlev\'e equations}}
\end{center}

\begin{center}

{\bf Adrian Stefan Carstea${}^\dagger$  and Tomoyuki Takenawa${}^\ddag$}\\
\medskip

${}^\dagger$ Department of Theoretical Physics, Institute of Physics and Nuclear Engineering\\ Reactorului 15, 077125, Magurele, Bucharest, Romania\\
E-mail: carstea@gmail.com\\[5pt]

${}^\ddag$ Faculty of Marine Technology, Tokyo University of Marine Science and Technology\\ 2-1-6 Etchu-jima, Koto-ku, Tokyo, 135-8533, Japan\\
E-mail: takenawa@kaiyodai.ac.jp

\end{center}
\begin{abstract}
A geometric study of two 4-dimensional mappings is given. By the resolution of indeterminacy they are lifted to pseudo-automorphisms 
of rational varieties obtained from $(\P^1)^4$ by blowing-up along sixteen 2-dimensional 
subvarieties. The symmetry groups,  the invariants and the degree growth rates are computed from the linearisation on the corresponding N\'eron-Severi bilattices.
It turns out that the deautonomised version of one of the mappings is a B\"acklund transformation of a direct product of the fourth Painlev\'e equation which has $A_2^{(1)}+A_2^{(1)}$ type affine Weyl group symmetry, while 
that of the other mapping is of Noumi-Yamada's $A_5^{(1)}$ Painlev\'e equation.
\end{abstract}

\section{Introduction}

The Painlev\'e equations are nonlinear second-order ordinary differential equations whose solutions are  meromorphic except some fixed points, but not reduced to known functions such as solutions of 
linear ordinary differential equations or Abel functions.
The discrete counterpart of Painlev\'e equations were introduced by Grammaticos, Ramani and their collaborators \cite{GRP1991, RGH1991} using so called the singularity confinement criterion. 
Since this criterion is not a sufficient condition for the mapping to be integrable, the notion of algebraic entropy was introduced by Hietarinta and Viallet \cite{HV1998} 
and studied geometrically in \cite{BV1999, Takenawa2001, Mase2018}. 
This entropy is essentially the same with topological entropy \cite{Gromov2003, Yomdin1987}.

Discrete Painlev\'e equations share many properties with the differential case, e.g., the existence of special solutions, such as algebraic solutions, or solutions expressed in terms of special functions, 
affine Weyl group symmetries and the geometric classification of equations in terms of rational surfaces.
Among them, associated families of rational surfaces, called the spaces of initial conditions, were introduced by Okamoto \cite{Okamoto1977} for the continuous case, and by Sakai \cite{Sakai2001} for the discrete case, 
where an equation gives a flow on a family of smooth projective rational surfaces. The cohomology group of the space of initial conditions gives information about the symmetries of the equation \cite{Sakai2001} and its degree growth
\cite{Takenawa2001}.

In recent years, research on four dimensional Painlev\'e systems has been progressed mainly from the viewpoint of isomonodromic deformation of linear equations \cite{Sakai2010,KNS2012},
while the space of initial conditions in Okamoto-Sakai's sense was known only for few equations. The difficulty lies in the part of using higher dimensional algebraic geometry. In the higher dimensional case the center of 
blowups is not necessarily a point but could be a subvariety of codimension two at least. 
Although some studies on symmetries of varieties or dynamical systems have been reported in the higher dimensional case, most of them consider only the case where varieties are obtained by blowups at points 
from the projective space \cite{DO1988, Takenawa2004, BK2008}. One of few exceptions is \cite{TT2009}, where
varieties obtained by blowups along codimension three subvarieties from the direct product of a projective line $(\P^1)^N$ were studied.

In this paper, starting with a mapping  $\varphi:\C^4\to\C^4;
(q_1, q_2, p_1, p_2)\mapsto (\bar{q}_1, \bar{q}_2, \bar{p}_1, \bar{p}_2)$:
\begin{align}\label{Map1}
A_2^{(1)}+A_2^{(1)}: &\left\{\begin{array}{rcl}
\bar{q}_1&=&-p_2-q_2+aq_2^{-1}+b\\
\bar{p}_1&=&q_2 \\
\bar{q}_2&=&-q_1-p_1+aq_1^{-1}+b\\
\bar{p}_2&=&q_1 
\end{array} \right.
\end{align}
and its slight modification: 
\begin{align} \label{Map2}
A_5^{(1)}:& \left\{\begin{array}{rcl} 
\bar{q}_1&=&-q_1-p_2+aq_2^{-1}+b_1\\
\bar{p}_1&=&q_2\\
\bar{q}_2&=&-q_2-p_1+aq_1^{-1}+b_2\\
\bar{p}_2&=&q_1\\
\end{array} \right.,
\end{align}
we construct their spaces of initial conditions, where the mappings are lifted to pseudo-automorphisms (automorphisms except finite number of subvarieties of codimension 2 at least).
We also give their deautonomisations and compute their degree growth.
It turns out that deautonomised version of mapping \eqref{Map1} is a B\"acklund transformation of a direct product of the fourth Painlev\'e equation, which has two continuous variables and $A_2^{(1)}+A_2^{(1)}$ (direct product) type symmetry, while 
that of mapping \eqref{Map2} is a B\"acklund transformation of Noumi-Yamada's $A_5^{(1)}$ Painlev\'e equation \cite{NY1998}, 
which has only one continuous variable and $A_5^{(1)}$ type symmetry. 
Although these equations might seem rather trivial compared to the Garnier systems, the Fuji-Suzuki system \cite{FS2010} or the Sasano system \cite{SY2007, Sasano2008}, 
we believe that they provide typical models for geometric studies on higher dimensional Painlev\'e systems.

The key tools of our investigation are pseudo-isomorphisms and N\'eron-Severi 
bilattices. In the autonomous case, for a given birational mapping, we successively blow-up a smooth projective rational variety along 
subvarieties to which a divisor is contracted. If this procedure terminate, the mapping is lifted to a pseudo-automorphism on a rational variety. 
In the non-autonomous case, the given sequence of 
mappings are lifted to a sequence of pseudo-isomorphisms between rational varieties.
We refer to those obtained rational varieties as the space of initial conditions (in Okamoto-Sakai's sense).  In this setting, the N\'eron-Severi bilattices play the role of root lattices of affine Weyl groups.\\

Let us make some remarks on the mappings.
Mapping \eqref{Map1} can be written in a simpler way as,
$$\overline{y_1}+y_2+\underline{y_1}-ay_2^{-1}-b=0$$
$$\overline{y_2}+y_1+\underline{y_2}-ay_1^{-1}-b=0,$$
where $y_1=q_1$, $y_2=q_2$, $\underline{y_1}=p_2$, $\underline{y_2}=p_1$
and the over/under bar denotes the image/preimage by the mapping.
As can be seen easily, when $y_1=y_2$, this system is one of the Quispel-Roberts-Thompson mappings\cite{QRT1989}.  
This fact enables us to find that Mapping \eqref{Map1} is the compatibility condition:
$$\overline{L}M-ML=0$$ 
for the Lax pair $L\Phi=h\Phi$, $ \overline{\Phi}=M\Phi$ with  
$$L=
\begin{bmatrix}
0&y_1&1&0&0&0\\
h&-a&b-y_1-\underline{y_2}&0&0&0\\
h\underline{y_2}&h&-a&0&0&0\\
0&0&0&0&y_2&1\\
0&0&0&h&-a&b-y_2-\underline{y_1}\\
0&0&0&h\underline{y_1}&h&-a
\end{bmatrix}
$$
and 
$$M=\begin{bmatrix}
0&0&0&a/y_2&1&0\\
0&0&0&0&0&1\\
0&0&0&h&0&0\\
a/y_1&1&0&0&0&0\\
0&0&1&0&0&0\\
h&0&0&0&0&0
\end{bmatrix},
$$
where $h$ is the spectral parameter.

Since $\overline{L}$ and $L$ are similar matrices, their characteristic polynomials are the same.
From the coefficients of the characteristic polynomial  ${\rm det}(x-L)$ with respect to $x$ and $h$, 
we have conserved quantities
\begin{align*}
I_1+I_2\mbox{ and } I_1I_2,
\end{align*}    
where
\begin{align}\label{I12}
&\begin{array}{rcl}
I_1&=& q_1p_1(q_1+p_1-b)-a(q_1+p_1)\\
I_2&=& q_2p_2(q_2+p_2-b)-a(q_2+p_2).\\
\end{array}
\end{align}

On the other hand, we do not know the Lax pair for Mapping \eqref{Map2}. However, using the space of initial conditions, we find two conserved quantities:
\begin{align}
I_1=& (q_1 p_1 -q_2 p_2)^2+b_1 b_2 (q_1 p_1+q_2 p_2)\nonumber \\&
      + b_1 \left( a (p_1 + q_2)-q_1 p_1^2-q_2^2p_2\right)  +b_2\left(a (q_1 + p_2)-q_1^2 p_1-q_2 p_2^2\right)  
    \nonumber\\
I_2=&
  (a (q_1 + p_2) + q_1 p_2 (b_2 - q_2 - p_1)) ( 
    a (q_2 + p_1)+ q_2 p_1 (b_1 - q_1 - p_2) ).  \label{I12.2}
\end{align}

Both mappings preserve the symplectic form $dq_1\wedge dp_1+dq_2\wedge dp_2 $ and thus the volume 
form $dq_1\wedge dp_1\wedge dq_2\wedge dp_2=\frac{1}{2}(dq_1\wedge dp_1+dq_2\wedge dp_2)^2$, whose coefficients give the canonical divisor class that corresponds 
to the above conserved quantities.   
Moreover, these conserved quantities $I_1$ and $I_2$ give the following continuous Hamiltonian flows:
\begin{align*}
 \frac{dq_1}{dt}=\frac{\partial I_i}{\partial p_1}, &\quad  \frac{dp_1}{dt}=- \frac{\partial I_i}{\partial q_1}\\ 
 \frac{dq_2}{dt}=\frac{\partial I_i}{\partial p_2}, &\quad  \frac{dp_2}{dt}=- \frac{\partial I_i}{\partial q_2}
\end{align*}
 commuting with each other:
\begin{align*}
 \{I_1,I_2\}= \left(\frac{\partial I_1}{\partial q_1}\frac{\partial I_2}{\partial p_1}
-\frac{\partial I_1}{\partial p_1}\frac{\partial I_2}{\partial q_1}\right)+
\left(\frac{\partial I_1}{\partial q_2}\frac{\partial I_2}{\partial p_2}
-\frac{\partial I_1}{\partial p_2}\frac{\partial I_2}{\partial q_2}\right)=0.
\end{align*}

As commented at the beginning,  another feature that indicates integrability of the mappings is the low degree growth rate,
where the degree of the $n$-th iterate of a mapping $\varphi^n$ is the degree with respect to the initial variables. 
Precise definition is given in the next section, but it is easily seem by numerical computation that the degree growth rate for both mappings
is quadratic.\\

This paper is organised as follows. 
In Section 2 we recall basic facts about the algebraic geometry used in this paper.
In Section 3 the singularity confinement test is applied for the above mappings.
In Section 4 we construct the spaces of initial conditions, where the mappings are lifted to  pseudo-automorphisms, and compute the actions on the Neron-S\'everi bilattices.
The degree growth is also computed for these actions.
In Section 5 symmetries of the spaces of initial conditions are studied. Deautonomised mappings are also given. 
In Section 6 we find Hamiltonians for continuous Painlev\'e equations defined on the above spaces of initial conditions.
Section 7 is devoted to discussion on generalisation of the results.\\

\noindent 
{\it Notation}: Throughout this paper, we often denote $x_{i_1}+x_{i_2}+\dots+x_{i_n}$ by $x_{i_1,i_2,\dots,i_n}$, where $x$ can be replaced by any symbols like $y, z, A, B, C$ etc.

\section{Algebraic stability and pseudo-isomorphisms}

A rational map $f:\P^n \to \P^n$ is given by $(n+1)$-tuple of homogeneous polynomials having the 
same degree (without common polynomial factor). Its degree, $\deg(f)$, is defined as the common degree of the $f_j$'s. We are interested 
in to compute $\deg(f^n)$, but it is not easy, since it only holds that $\deg(f^n)\leq (\deg f)^n$ in general by cancellation of common factors.
A related object is the indeterminacy set of $f$ given by
$$I(f)=\{{\bf x}\in\P^n~|~f_0({\bf x})=\dots=f_n({\bf x})=0\}$$
that is a subvariety of codimension 2 at least, whereas $f$ defines a holomorphic mapping $f:\P^n\setminus I(f)\to\P^n$. 
In this section we recall basic facts in algebraic geometry used in this kind of study.\\

\noindent {\it Rational correspondence}\\
Let $\cX$ and $\cY$ be smooth projective varieties of dimension $N$ and $f:\cX \to \cY$ a dominant rational map. Using the completion of the graph of $f$, $G_f$, we can decompose $f$ as $f=\pi_{\cY} \circ \pi_{\cX}^{-1}$ such that
$\pi_{\cX}: G_f \to \cX$ and $\pi_{\cY}: G_f \to \cY$ are rational morphisms 
and the equality holds for generic points in $\cX$. 
 
This definition is simple but practically may arise complications in computing defining polynomials of the graph.
For example, when $\cX$ and $\cY$ are rational varieties and 
$(x_1,\dots,x_N)$ and $(y_1,\dots,y_N)$ are their local coordinates,
introducing homogeneous coordinates as $(X_0:\dots:X_N)=(X_0:X_0x_1:\dots:X_0x_N)$ and 
$(Y_0:\dots:Y_N)=(Y_0:Y_0y_1:\dots:Y_0y_N)$, we can only say that the graph $G_f$ is ``one of the components'' of $Y_k p_l(X_0,\cdots,X_N) =Y_l p_k(X_0,\cdots,X_N)$, $k,l=0,\dots,N$, where 
$(y_1,\dots,y_N)=(p_1/p_0,\dots,p_N/p_0)$ is the induced homogeneous map and
$(X_0:\dots:X_N; Y_0:\dots:Y_N)$  is the coordinate system of $\P^N\times \P^N$ (\S5 of \cite{Bedford2003} and Example 3.4 of \cite{Roeder2015} are examples of such complication).

Hironaka's singularity resolution theorem (Question (E) in \S~0.5 of \cite{Hironaka1964}) also gives this decomposition in a more tractable form as: there exists a sequence of blowups $\pi:\tilde{\cX}\to \cX$
along smooth centers in $I(f)$ such that the induced rational map $\tilde{f} : \tilde{\cX} \to \cY$
is a morphism.
 
Using these decompositions we can define the push-forward and the pull-back correspondence of a sub-variety by $f$ as $f_c(V)=\pi_{\cY} \circ \pi_{\cX}^{-1} (V)=\tilde{f} \circ \pi^{-1} (V)$ for $V \subset \cX$ and $f_c^{-1}(W)=\pi_{\cX} \circ \pi_{\cY}^{-1} (W)=\pi \circ \tilde{f}^{-1} (W)$ for $W\subset \cY$. We denote their restriction to divisor groups by $f_*:\div(\cX)\to \div(\cY)$ and $f^*:\div(\cY)\to \div(\cX)$, where lower dimensional subvarieties are ignored as zero divisors.
Especially, when $f$ is birational, it obviously holds that $f_*=(f^{-1})^*$ and $f^*=(f^{-1})_*$. \\
  
\noindent {\it Algebraic stability}\\
The following proposition is fundamental to our study. Its two dimensional version was shown by
Diller and Favre (Proposition~1.13 of [DF01]) . ``If'' part was shown by Bedford-Kim (Theorem~1.1 of \cite{BK2008}) and  Roeder (Proposition~1.5 of \cite{Roeder2015}), while ``only if'' part by Bayraktar (Theorem~5.3 of \cite{Bayraktar2012}). 

\begin{proposition}\label{AS1}
Let $f:\cX\to \cY$ and $g:\cY\to \cZ$ be dominant rational maps. Then
$f^*\circ g^* =  (g\circ f)^*$ 
holds if and only if there does not exist a prime divisor $D$ on $\cX$ such that $f(D\setminus I(f))\subset I(g)$.
\end{proposition} 

Since the proof of ``if'' part is very simple, it would be convenient to quote from \cite{BK2008}, modifying it to fit our terminologies:
\begin{quote} 
If $D$ is a divisor on $\cZ$ then $g^*(D)$ is a divisor on $\cY$ which is the same as $g_c^{-1}(D)$ on $\cY-I(g)$ by ignoring codimension greater than one. Since $I(g)$ 
has codimension at least 2 we also have $(g\circ f)^*(D)=f^*(g^*(D))$ on $\cX-I(f)-f_c^{-1}(I(g))$. By the hypothesis $f_c^{-1}(I(g))$ has codimension at least 2. Thus we have $(g\circ f)^*(D)=f^*g^*(D)$ on $\cX$.
\end{quote}

\begin{example}\label{EX1}
Let $(x_0:x_1,x_2:x_3)$ be the homogeneous coordinate system of the complex projective space $\P^3$. 
Let $\cX$ be a variety obtained by blowing up $\P^3$ along the line $x_1=x_2=0$,
$\cY$ be  $\P^3$, $\cZ$ be a variety obtained by blowing up $\P^3$ at the point $x_1=x_2=x_3=0$, and $f:\cX\to \cY$ and $g:\cY\to \cZ$ be the identity map on $\P^3$. 
Let $H$, $E_{\cX}$ and $E_{\cZ}$ denote the class of the total transform of the hyperplane, the exceptional divisors of $\cX$ and $\cZ$ respectively. 
Then it holds that $I(f)=\emptyset$, $I(g)=\{(1:0:0:0)\}$ and there is no prime divisor $D \in \cX$ such that $f(D) \subset I(g)$, while    
$I(f^{-1})=\{(s:0:0:t)~|~(s:t)\in \P^1 \}$, $I(g^{-1})=\emptyset$ and $g^{-1}(E_{\cZ}) \subset I(f^{-1})$.
Thus, $f^*g^*=(g\circ f)^*$ holds, but not $(g^{-1})^*(f^{-1})^*=(f^{-1}\circ f^{-1})^*$  (see Fig.~\ref{fig.EX1}).

The pull-backs acts on divisor classes as 
\begin{align*}
f^*:& H \mapsto H\\
g^*:& H \mapsto H, \quad E_{\cZ} \mapsto 0\\
f^* g^*:& H \mapsto H, \quad E_{\cZ} \mapsto 0\\
(g\circ f)^*:& H \mapsto H,\quad E_{\cZ} \mapsto 0\\
(f^{-1})^*:& H \to H,\quad E_{\cX} \to 0\\
(g^{-1})^*:& H \to H\\
(g^{-1})^* (f^{-1})^*:& H \mapsto H, \quad E_{\cX} \mapsto 0\\
(f^{-1}\circ g^{-1})^*:& H \mapsto H,\quad E_{\cX} \mapsto E_{\cZ}.
\end{align*}
In particular, for the anti-canonical divisor classes $-K_{\cX}=4 H- E_{\cX}$, 
 $-K_{\cY}=4 H$ and  $-K_{\cZ}=4 H- 2E_{\cZ}$,
$(g^{-1})^* (f^{-1})^*(-K_{\cX})=4 H$ is greater than $(f^{-1}\circ g^{-1})^*(-K_{\cX})=4 H- E_{\cZ}$. 
\end{example}

\begin{figure}
\begin{center}
\begin{tikzpicture}[line cap=round,line join=round,>=triangle 45,x=0.6cm,y=0.6cm]
\clip(-2,-3) rectangle (22,6);
\draw [line width=1.pt] (0.98,-1.84)-- (5.,0.);
\draw [line width=1.pt] (1.0134970368153247,3.0003218198144275)-- (1.6670448368751865,3.299458126309488);
\draw [line width=1.pt] (1.6670448368751865,3.299458126309488)-- (-0.8619814426818323,0.);
\draw [line width=1.pt] (-0.8619814426818323,0.)-- (-1.5155292427416949,-0.29913630649506107);
\draw [line width=1.pt] (-1.5155292427416949,-0.29913630649506107)-- (1.0134970368153247,3.0003218198144275);
\draw [line width=1.pt] (0.98,-1.84)-- (-1.5155292427416949,-0.29913630649506107);
\draw [line width=1.pt] (1.0134970368153247,3.0003218198144275)-- (0.98,-1.84);
\draw [line width=1.pt] (5.,0.)-- (-0.8619814426818323,0.);
\draw [line width=1.pt] (5.,0.)-- (1.6670448368751865,3.299458126309488);
\draw [->,line width=2.pt] (4.5,2.) -- (6.5,2.);
\draw [line width=1pt] (-2+8.5,0.)-- (0.98+8.5,-1.84);
\draw [line width=1pt] (0.98+8.5,-1.84)-- (1.02+8.5,3.94);
\draw [line width=1pt] (1.02+8.5,3.94)-- (5+8.5,0.);
\draw [line width=1pt] (5+8.5,0.)-- (-2.+8.5,0.);
\draw [->,line width=1.pt] (-2.+8.5,0.) -- (-0.9888551724137928+8.5,-0.6243310344827588);
\draw [->,line width=1.pt] (-2.+8.5,0.) -- (-0.26862385321100923+8.5,0.);
\draw [line width=1.pt] (1.02+8.5,3.94)-- (-2.+8.5,0.);
\draw [line width=1.pt] (0.98+8.5,-1.84)-- (5.+8.5,0.);
\draw [->,line width=1.pt] (-1.9636286317156306+8.5,0.04745138776172701) -- (-1.226250608667424+8.5,1.009461126440513);
\begin{scriptsize}
\draw[color=black] (-1.41+8.5,-1) node {$x_1/x_0$};
\draw[color=black] (-0.3+8.5,0.5) node {$x_2/x_0$};
\draw[color=black] (-2+8.5,1.2) node {$x_3/x_0$};
\draw[color=black] (-2+8.5,1.2) node {$x_3/x_0$};
\end{scriptsize}
\draw[color=black] (5.3,2.5) node {$f$};
\draw[color=black] (5.3+8.5,2.5) node {$g$};
\draw[color=black] (-1, 1.5) node {$E_{{\mathcal X}}$};
\draw[color=black] (-2+17, 0) node {$E_{{\mathcal Z}}$};
\draw [->,line width=2.pt] (4.5+8.5,2.) -- (6.5+8.5,2.);
\draw [line width=1pt] (0.98+17,-1.84)-- (1.02+17,3.94);
\draw [line width=1pt] (1.02+17,3.94)-- (5.+17,0.);
\draw [line width=1.pt] (0.98+17,-1.84)-- (5.+17,0.);
\draw [line width=1.pt] (1.02+17,3.94)-- (-1.2143392306443759+17,1.0250011361791915);
\draw [line width=1.pt] (-0.17893199155981188+17,0.)-- (-1.2247453335497485+17,-0.4786807336471354);
\draw [line width=1.pt] (-1.2143392306443759+17,1.0250011361791915)-- (-1.2247453335497485+17,-0.4786807336471354);
\draw [line width=1.pt] (-1.2143392306443759+17,1.0250011361791915)-- (-0.17893199155981188+17,0.);
\draw [line width=1.pt] (-1.2247453335497485+17,-0.4786807336471354)-- (0.98+17,-1.84);
\draw [line width=1.pt] (-0.17893199155981188+17,0.)-- (5.+17,0.);
\end{tikzpicture}
\caption{Example~\ref{EX1}}
    \label{fig.EX1}
\end{center}
\end{figure}

A rational map $\varphi$ from a smooth projective variety $\cX$ to itself 
is called \textit{algebraically stable} or \textit{1-regular} if 
$(\varphi^*)^n=(\varphi^n)^*$ holds \cite{FS1995}. 
The following proposition is obvious from Proposition~\ref{AS1}.
 
\begin{proposition}\label{AS2}
A rational map $\varphi$ from a smooth projective variety $\cX$ to itself  
is algebraically stable if and only if
there does not exist a positive integer $k$ and a divisor $D$ on $\cX$ such that $f(D\setminus I(f))\subset I(f^k)$.
\end{proposition}

\noindent {\it Pseudo-isomorphisms and N\'eron-Severi bilattices}\\
For a smooth projective variety $\cX$, 
the N\'eron-Severi lattice $N^1(\cX) =  \pic(\cX)/\pic^0(\cX) \subset H^2(\mathcal{X},\Z)$, where $\pic^0(\cX)$ is the connected comonent of the Picard group, is the fist Chern class of the Picard group $c_1:\pic(\cX)\to H^2(\cX,\Z) $. This lattice and its Poincar\'e dual $N_1(\cX) \subset H_2(\mathcal{X},\Z)$ are finitely generated lattices.

We call a birational mapping $\varphi:\cX\to \cY$ a {\it pseudo-isomorphism}
if $\varphi$ is isomorphic except on finite number of subvarieties of codimension two at least. This conditions is equivalent to that there is no prime divisor pulled back to zero divisor by $f$ or $f^{-1}$. Hence, if $\varphi$ is a pseudo-automorphism, then $\varphi$ and $\varphi^{-1}$ are algebraically stable.  

\begin{proposition}[\cite{DO1988}]\label{NSauto}
Let $\cX$ and $\cY$ be smooth projective varieties and
$\varphi$ a pseudo-isomorphism from $\cX$ to $\cY$. Then
$\varphi$ acts on the N\'eron-Severi bi-lattice as an automorphism preserving the intersections.
\end{proposition}

\begin{proof}
It is obvious that $\varphi_*:N^1(\cX)\mapsto N^1(\cY)$ is an isomorphism by definition of pseudo-isomorphisms. The action $\varphi_*:N_1(\cX)\mapsto N_1(\cY)$ is determined by this isomorphism and the Poincar\'e duality. 
\end{proof}

\medskip

\noindent {\it Blowup of a direct product of $\P^m$}\\
As we have seen in the example, it is convenient to write the generators of 
the N\'eron-Severi bilattice explicitly. Following \cite{TT2009}, we give some formulae 
for some rational varieties which appear as spaces of initial conditions of Painlev\'e systems.
Note that the N\'eron-Severi bilattice coincides with $H^2(\cX,\Z)\times H_2(\cX,\Z)$ if $\cX$ is a smooth projective rational variety, since $\pic^0(\cX)=\{0\}$ in this case.

Let $\cX$ be a rational variety obtained by $K$ successive blowups from $\P^{m_1}\times\cdots \times\P^{m_n}$ with $N=m_1+\cdots+m_n$, and $({\bf x}_1,\dots, {\bf x}_n)$ its coordinate chart with homogeneous coordinates ${\bf x}_i=(x_{i0}:x_{i1}:\cdots:x_{im_i})$. Let $H_i$ denote the total transform of the class of a hyper-plane ${\bf c}_i \cdot {\bf x}_i=0$, where ${\bf c}_i$ is a constant vector in $\P^{m_i}$, and $E_k$ the total transform of the $k$-the exceptional divisor class. Let $h_i$ denote the total transforms of the class of a line
$$\{{\bf x}~|~ {\bf x}_j={\bf c}_j (\forall j\neq i),\ {\bf x}_i=s{\bf a}_i+t{\bf b}_i (\exists (s:t)\in \P^1) \},$$ 
where ${\bf a}_i$,  ${\bf b}_i$
and ${\bf c}_j$'s are constant vectors in $\P^{m_i}$ and $\P^{m_j}$ respectively, and $e_k$ the class of a line in a fiber of the $k$-th blow-up. Note that the exceptional divisor for a blowing-up along a $d$-dimensional subvariety $V$ is isomorphic to $V\times \P^{N-d-1}$, where $\P^{N-d-1}$ is a fiber.   

Then the Picard group $\simeq H^2(\mathcal{X},\Z)$ and its Poincar\'e dual $\simeq H_2(\mathcal{X},\Z)$ are lattices
\begin{align}\label{NSbasis}
H^2(\mathcal{X},\Z)=\bigoplus_{i=1}^n \Z H_i \oplus \bigoplus_{k=1}^{K} \Z E_k,\quad 
H_2(\mathcal{X},\Z)=\bigoplus_{i=1}^n \Z h_i \oplus \bigoplus_{k=1}^{K} \Z e_k
\end{align}
and the intersection form is given by
\begin{align}
\langle H_i, h_j\rangle = \delta_{ij},\quad 
\langle E_l, e_l\rangle = -\delta_{kl},\quad
\langle H_i, e_k\rangle =0.
\end{align}

Let $\varphi$ be a pseudo-automorphism on $\cX$, and $A$ and $B$ be matrices representing
$\varphi_*: H^2(\cX,\Z)\to  H^2(\cY,\Z)$ and $\varphi_*: H_2(\cX,\Z)\to  H_2(\cY,\Z)$
respectively on basis \eqref{NSbasis}.  
Then, for any ${\bf f}\in H^2(\cX,\Z)$ and ${\bf g}\in  H_2(\cY,\Z)$ it holds that
\begin{align*}
 \langle {\bf f}, {\bf g}\rangle= {\bf f}^T J  {\bf g}, \quad
J=\begin{bmatrix}
I_n&0\\0& -I_K
\end{bmatrix},
\end{align*}
where ${*}^T$ denotes transpose and $I_m$ denotes the identity matrix of size $m$. 
Thus, $\langle A{\bf f}, B{\bf g}\rangle= \langle {\bf f}, {\bf g}\rangle$ yields $A^T J B = J$, and 
hence 
\begin{align}\label{H2H2}
B= J (A^{-1})^T J,
\end{align} 
which is a formula for computing the action on $H_2(\cX,\Z)$ from that on $H^2(\cX,\Z)$.

\begin{example}\label{SCT}
Let $\cX$ be obtained by blowing up $\P^3$ at four points $(1:0:0:0)$, $(0:1:0:0)$, $(0:0:1:0)$, $(0:0:0:1)$, 
and both $f:\cX\to \cX$ be the standard Cremona transformation of $\P^3$:
$(x_0:x_1:x_2:x_3)\to (x_0^{-1}:x_1^{-1}:x_2^{-1}:x_3^{-1})$. 
 Then $I(f)$ consists of the proper (strict) transform of 6 lines passing through two of the four points blown up.
 This is a simple example of a pseudo-automorphism  (see Fig.~\ref{fig.SCT}). 

The push-forward action on divisor classes is 
\begin{align*}
&f_*: H \mapsto 3H -2E_{0,1,2,3},\quad E_i\mapsto H-E_{i+1,i+2,i+3} \quad (i=0,1,2,3 \mod 4),
\end{align*}
where $E_{i_1,\dots,i_k}=E_{i_1}+\cdots+E_{i_k}$,
while its dual is
\begin{align*}
&f_*: h \mapsto 3h -e_{0,1,2,3},\quad e_i\mapsto 2h-e_{i+1,i+2,i+3} \quad (i=0,1,2,3 \mod 4).
\end{align*}
The corresponding representing matrices 
\begin{align*}
A=\begin{bmatrix}
3&1& 1& 1& 1\\-2& 0& -1& -1& -1\\-2& -1& 0& -1& -1\\
-2&-1&-1&0& -1\\-2& -1& -1& -1& 0
 \end{bmatrix},\quad
B=\begin{bmatrix}
3&2& 2& 2& 2\\-1& 0& -1& -1& -1\\-1& -1& 0& -1& -1\\
-1&-1&-1&0& -1\\-1& -1& -1& -1& 0
 \end{bmatrix}
\end{align*}
satisfies \eqref{H2H2}.
 It is also easy to check that $(f_*)^2$ is the identity as it should be.
\end{example}

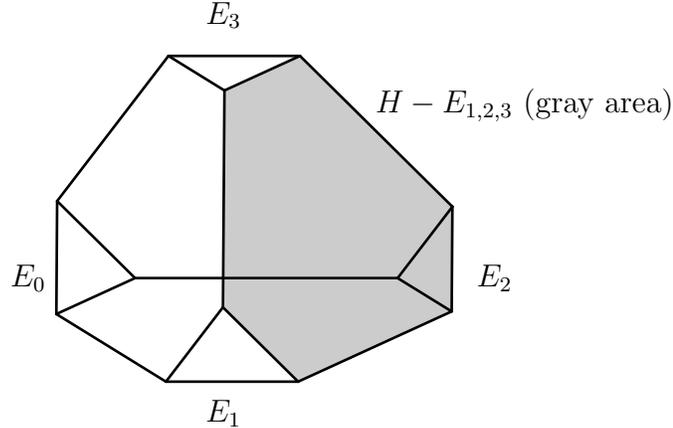
\begin{figure}[htbp]
\begin{center}
\begin{tikzpicture}[line cap=round,line join=round,>=triangle 45,x=1.0cm,y=1.0cm]
\clip(-4,-2.2) rectangle (7.5,4.0);
\fill[color=light-gray] (1.0100074663204026,2.496078883298166) -- (0.9900427197130277,-0.3888270014674707) -- (1.9892933311593026,-1.3780348932007174) -- (4.030806917016269,-0.4436107643507623) -- (4.040450629284764,0.9499056584467414) -- (2.014257101119948,2.9557354325596488) -- cycle;
\draw [line width=1.pt] (-0.17893199155981188,0.)-- (-1.2247453335497485,-0.4786807336471354);
\draw [line width=1.pt] (-1.2143392306443759,1.0250011361791915)-- (-1.2247453335497485,-0.4786807336471354);
\draw [line width=1.pt] (-1.2143392306443759,1.0250011361791915)-- (-0.17893199155981188,0.);
\draw [line width=1.pt] (0.23181738137942265,-1.3780348932007174)-- (0.9900427197130277,-0.3888270014674707);
\draw [line width=1.pt] (0.9900427197130277,-0.3888270014674707)-- (1.9892933311593026,-1.3780348932007174);
\draw [line width=1.pt] (1.9892933311593026,-1.3780348932007174)-- (0.23181738137942265,-1.3780348932007174);
\draw [line width=1.pt] (4.030806917016269,-0.4436107643507623)-- (3.312350353013404,0.);
\draw [line width=1.pt] (3.312350353013404,0.)-- (4.040450629284764,0.9499056584467414);
\draw [line width=1.pt] (4.040450629284764,0.9499056584467414)-- (4.030806917016269,-0.4436107643507623);
\draw [line width=1.pt] (0.26556370719039113,2.955735432559649)-- (1.0100074663204026,2.496078883298166);
\draw [line width=1.pt] (1.0100074663204026,2.496078883298166)-- (2.014257101119948,2.9557354325596488);
\draw [line width=1.pt] (2.014257101119948,2.9557354325596488)-- (0.26556370719039113,2.955735432559649);
\draw [line width=1.pt] (0.26556370719039113,2.955735432559649)-- (-1.2143392306443759,1.0250011361791915);
\draw [line width=1.pt] (-1.2247453335497485,-0.4786807336471354)-- (0.23181738137942265,-1.3780348932007174);
\draw [line width=1.pt] (1.9892933311593026,-1.3780348932007174)-- (4.030806917016269,-0.4436107643507623);
\draw [line width=1.pt] (4.040450629284764,0.9499056584467414)-- (2.014257101119948,2.9557354325596488);
\draw [line width=1.pt] (1.0100074663204026,2.496078883298166)-- (0.9900427197130277,-0.3888270014674707);
\draw [line width=1.pt] (-0.17893199155981188,0.)-- (3.312350353013404,0.);
\draw[color=black] (-1.6, 0) node {$E_0$};
\draw[color=black] (1, -1.8) node {$E_1$};
\draw[color=black] (4.6, 0) node {$E_2$};
\draw[color=black] (1, 3.5) node {$E_3$};
\draw[color=black] (5, 2.3) node {$H-E_{1,2,3}$ (gray area)};
\end{tikzpicture}
\caption{Example~\ref{SCT}: $E_0$ and $H-E_{1,2,3}$ are exchanged.}
    \label{fig.SCT}
\end{center}
\end{figure}

\medskip

\noindent {\it Degree of a mapping}\\
Let $\varphi$ be a rational mapping from $\C^N$ to itself: 
$$ \varphi: (\bar{x}_1,\dots,\bar{x}_N)= (\varphi_1(x_1,\cdots,x_N),\dots,\varphi_N(x_1,\cdots,x_N)).$$ 
The degree of $\bar{x_i}$ of $\varphi$ with respect $x_j$ is defined as
the degree of $\varphi_i$ as a rational function of $x_j$, i.e. the maximum of degrees of
numerator and denominator.    
Let $\cX$ be a rational variety obtained by $K$ successive blowups from $(\P^1)^N$.
Then the degree of $\bar{x_i}$ of $\varphi$ with respect $x_j$ is given by
the coefficient of $H_j$ in $\varphi^*(H_i)$. When $\varphi$ is iterated,   
the degree of $\bar{x_i}$ of $\varphi^n$ with respect $x_j$ is given by
the coefficient of $H_j$ in $(\varphi^n)^*(H_i)$, which coincides with $(\varphi^*)^n(H_i)$
if $\varphi$ is algebraically stable on $\cX$. (The reason is exactly the same with the two-dimensional case. See \cite{Takenawa2001} for details.)

There is another (and more standard) definition of the mapping degree. 
Let $\varphi$ be a rational mapping on $\C^N$ as above.
We can extend the action of $\varphi$ onto $\P^N$ 
by replacing $x_j$ by $x_j/x_0$, rewriting $\varphi_i$'s so that they have the common denominator
and considering them as $\bar{x}_i/\bar{x}_0$.
Then $\varphi$ can be expressed as
$$ \varphi: (\bar{x}_0: \dots:\bar{x}_N)= (p_0(x_0,\dots,x_N): \dots:p_N(x_0,\cdots,x_N)),$$
where $p_i$'s are homogeneous polynomials and the common factor is only a constant. 
Then, the degree of $\varphi$ is defined as the common degree of $p_i$'s.
Let $\cX$ be a rational variety obtained by $K$ successive blowups from $\P^N$.
Then the degree of $\varphi$ is given by the coefficient of $H$ in $\varphi^*(H)$. 
When $\varphi$ is iterated, the degree of $\varphi^n$ is given by
the coefficient of $H$ in $(\varphi^n)^*(H)$, which coincides with $(\varphi^*)^n(H)$
if $\varphi$ is algebraically stable on $\cX$.

Above two kinds of degrees are related to each other. Indeed, it is clearly holds that
\begin{align*} &\max_i \{ \mbox{ $\sum_j$ degree of $\varphi_i$ for $x_j$}\} \leq \mbox{degree of $\varphi$}\\
&\qquad \leq N \max_i \{ \mbox{ $\sum_j$ degree of $\varphi_i$ for $x_j$}\}.
\end{align*}

Of course we can also consider intermediate of the above degrees by extending the action of $\varphi$ onto $\P^{m_1}\times\cdots \times\P^{m_n}$ with $N=m_1+\cdots+m_n$.
But we do not use such degrees in this paper and omit them.

\section{Singularity confinement}

The idea of the singularity confinement test is as follows. Consider a hypersurface in some compactification $X$ of $\C^n$ which is contracted to a lower dimensional variety (singularity) by a birational automorhism $f$ of $X$. We say the singularity to be confined if there exists an integer $n\geq 2$ such that the hypersurface is recovered to some hypersurface by $f^n$ in generic. In this case the memory of initial 
conditions is said to be recover. Let us introduce the set of contracted hypersurfaces: 
$${\mathcal E}(f)=\{D \subset  X:\mbox{hypersurface}~|~{\rm det}({\partial f}/{\partial x})=0 \mbox{ on $
D$ in generic}\},$$
where zero of the Jacobian contraction to a lower dimensional variety. If singularity is confined for every $D$ in $\cE(f)$, we say that the initial data is not lost and the map $f$ satisfies the singularity confinement criterion. 
Note that the existence of confined singular sequence implies algebraical unstability.   

In this section we consider the mappings on compactified space $(\P^1)^4=(\C\P^1)^4$ and apply the singularity confinement test to them.\\

\noindent {\it Case $A_2^{(1)} + A_2^{(1)}$}:\\
If we take $q_1=\ve$ with $|\ve| \ll 1$ and the others are generic, the principal terms  of the Laurent series with respect to $\ve$ in the trajectories are
\begin{align*}
&(\ve, p_1^{(0)}, q_2^{(0)}, p_2^{(0)})\mbox{: 3 dim}\\
&\qquad \to   (q_1^{(1)}, p_1^{(1)}, a\ve^{-1}, \ve)\mbox{: 2 dim}\ \fbox{14}\\
&\qquad \to (-a\ve^{-1}, a\ve^{-1}, q_2^{(2)}, p_2^{(2)})\mbox{: 2 dim}\ \fbox{4}\\
&\qquad \to (q_1^{(3)}, p_1^{(3)},-\ve,-a\ve^{-1})\mbox{: 2 dim}\ \fbox{16}\\
&\qquad \to ( q_1^{(4)},-\ve, q_2^{(4)}, p_2^{(4)})\mbox{: 3 dim},
\end{align*}
where $x_i^{(j)}$ denotes a generic value in $\C$, 
``$k$ dim'' denotes the dimension of corresponding subvariety in $(\P^1)^4$ and 
$\fbox{$n$}$ denotes the order of blowing up that we explain in the next section.
Similarly, starting with $q_2=\ve$ and the others being generic, we get
\begin{align*}
&(q_1^{(0)}, p_1^{(0)}, \ve, p_2^{(0)})\mbox{: 3 dim}\\
&\qquad \to  (a\ve^{-1}, \ve, q_2^{(1)}, p_2^{(1)})\mbox{: 2 dim}\ \fbox{6}\\
&\qquad \to (q_1^{(2)}, p_1^{(2)}, -a\ve^{-1}, a\ve^{-1})\mbox{: 2 dim}\ \fbox{12} \\
&\qquad \to (-\ve, -a\ve^{-1}, q_2^{(3)}, p_2^{(3)})\mbox{: 2 dim}\ \fbox{8}\\
&\qquad \to (q_1^{(4)}, p_1^{(4)}, q_2^{(4)},-\ve)\mbox{: 3 dim}.
\end{align*}
In both two cases, information on the initial values $x_i^{(0)}$ is recovered after finite number of steps, and thus singularities are confined.
   
We also find another (cyclic) singularity pattern as
\begin{align*}
&(\ve^{-1}, p_1^{(0)}, q_2^{(0)}, p_2^{(0)})\mbox{: 3 dim}\\
&\qquad \to (q_1^{(1)}, p_1^{(1)},-\ve^{-1}, \ve^{-1})\mbox{: 2 dim}\ \fbox{10}\\
&\qquad \to (q_1^{(2)},-\ve^{-1}, q_2^{(2)}, p_2^{(2)})\mbox{: 3 dim}\\
&\qquad \to (q_1^{(3)}, p_1^{(3)}, \ve^{-1}, p_2^{(3)})\mbox{: 3 dim}\\
&\qquad \to (-\ve^{-1}, \ve^{-1}, q_2^{(4)}, p_2^{(4)})\mbox{: 2 dim}\ \fbox{2}\\
&\qquad \to (q_1^{(5)}, p_1^{(5)}, q_2^{(5)},-\ve^{-1})\mbox{: 3 dim}\\
&\qquad \to(\ve^{-1}, p_1^{(6)}, q_2^{(6)}, p_2^{(6)}),
\end{align*}
where the last hyper-surface is the same with the first one.

Moreover, since we need several times blowups for resolve each singularity, we should consider the following singularity sequences as well, where base varieties of those blow-ups appear
\begin{align*}
&(c_1^{(0)}\ve^{-1},  c_2^{(0)} \ve^{-1}, q_2^{(0)}, p_2^{(0)})\mbox{: 2 dim}\ \fbox{1}\\
&\qquad \to (q_1^{(1)}, p_1^{(1)}, c_1^{(1)}\ve^{-1},  c_2^{(1)}\ve^{-1})  \mbox{: 2 dim}\ \fbox{9}\\
&\qquad \to (c_1^{(2)}\ve^{-1},  c_2^{(2)} \ve^{-1}, q_2^{(2)}, p_2^{(2)})\\
&(c_1^{(0)}\ve^{-1},  c_2^{(0)} \ve, q_2^{(0)}, p_2^{(0)})\mbox{: 2 dim}\ \fbox{5}\\
&\qquad \to (q_1^{(1)}, p_1^{(1)}, c^{(1)}\ve^{-1}, c^{(1)}\ve^{-1})  \mbox{: 2 dim}\ \fbox{11} \\
&\qquad \to (c_1^{(2)}\ve,  c_2^{(2)} \ve^{-1},  q_2^{(2)}, p_2^{(2)})\mbox{: 2 dim}\ \fbox{7}\\
&\qquad \to (q_1^{(3)}, p_1^{(3)}, c_1^{(3)} \ve^{-1}, c_2^{(3)}\ve)\mbox{: 2 dim}\ \fbox{13}\\
&\qquad \to (c^{(4)}\ve^{-1}, c^{(4)}\ve^{-1}, q_2^{(4)}, p_2^{(4)})\mbox{: 2 dim}\ \fbox{3}\\
&\qquad \to  ( q_1^{(5)}, p_1^{(5)}, c_1^{(5)}\ve,  c_2^{(5)} \ve^{-1})\mbox{: 2 dim}
\ \fbox{15}\\
&\qquad \to (c_1^{(6)}\ve^{-1},  c_2^{(6)} \ve, q_2^{(6)}, p_2^{(6)}).
\end{align*}
where the last subvariety for each sequence is the same with the first one.

The inclusion relations of these bases of blow-ups are
\begin{align}
&{\color{black} \fbox{1}}\supset {\color{black} \fbox{2}} 
\supset {\color{black} \fbox{3}} 
\supset {\color{black} \fbox{4}},\quad 
 {\color{black} \fbox{5}} 
\supset {\color{black} \fbox{6}},\quad 
 {\color{black} \fbox{7}} 
\supset {\color{black} \fbox{8}}\nonumber\\
&{\color{black} \fbox{9}}\supset {\color{black} \fbox{10}} 
\supset {\color{black} \fbox{11}} 
\supset {\color{black} \fbox{12}},\quad 
 {\color{black} \fbox{13}} 
\supset {\color{black} \fbox{14}},\quad 
 {\color{black} \fbox{15}} 
\supset {\color{black} \fbox{16}}, \label{inclusion}
\end{align}
where we need to compare lower terms of the Laurent series to see these relations.\\

\noindent {\it Case $A_5^{(1)}$}:\\
We find following two singularity sequences:
\begin{align*}
&(\ve,p_1^{(0)},q_2^{(0)},p_2^{(0)})\mbox{{\color{black}: 3 dim}}\\
&\qquad \to (-p_2^{(0)}+a/q_2^{(0)}+b_1,q_2^{(0)}, a\ve^{-1},\ve)\mbox{{\color{black}: 2 dim}} \ {\color{black} \fbox{6}} \\
&\qquad \to  (p_2^{(0)}-a/q_2^{(0)},a\ve^{-1},-a\ve^{-1},-p_2^{(0)}+a/q_2^{(0)}+b_1)\mbox{{\color{black}: 1 dim}}  \ {\color{black} \fbox{4}}\\
&\qquad \to (-\ve, -a\ve^{-1}, q_2^{(3)},p_2^{(0)}-a/q_2^{(0)})\mbox{{\color{black}: 2 dim}}  \ {\color{black} \fbox{8}}\\
&\qquad \to  (q_1^{(4)},p_1^{(4)},q_2^{(4)},-\ve)\mbox{{\color{black}: 3 dim}},
\end{align*}
and
\begin{align*}
&(q_1^{(0)},p_1^{(0)},\ve^{-1},p_2^{(0)})\mbox{{\color{black}: 3 dim}}\\
&\qquad \to  (-p_2^{(0)}-q_1^{(0)}+b_1,\ve^{-1}, -\ve^{-1},q_1^{(0)})\mbox{{\color{black}: 2 dim}} \ {\color{black} \fbox{2}}\\
&\qquad \to  (p_2^{(0)}, -\ve^{-1},q_2^{(2)},-p_2^{(0)}-q_1^{(0)}+b_1)\mbox{{\color{black}: 3 dim}}\\
&\qquad \to  (q_1^{(3)},p_1^{(3)},\ve^{-1},p_2^{(3)}) \mbox{Returned}.
\end{align*}

We should consider the following singularity sequences as well, where base varieties of those blow-ups appear.
\begin{align*}
&(q_1^{(0)},c_1^{(0)}\ve^{-1},c_2^{(0)} \ve^{-1},p_2^{(0)})\mbox{{\color{black}: 2 dim}} \ {\color{black} \fbox{1}}\\
&\qquad \to (q_1^{(1)},c_1^{(1)} \ve^{-1},c_2^{(1)} \ve^{-1},p_2^{(1)})  \mbox{{\color{black}: Returned}}\\
&(q_1^{(0)},p_1^{(0)},c_1^{(0)} \ve^{-1},c_2^{(0)}\ve )\mbox{{\color{black}: 2 dim}} \ {\color{black} \fbox{5}}\\
&\qquad \to
(-q_1^{(0)}+b_1,c_1^{(0)} \ve^{-1},-c_1^{(0)} \ve^{-1},q_1^{(0)})\mbox{{\color{black}: 1 dim}}  \ {\color{black} \fbox{3}}\\
 &\qquad \to 
(c_2^{(2)}\ve,-c_1^{(0)} \ve^{-1},p_2^{(0)},-q_1^{(0)}+b_1)\mbox{{\color{black}: 2 dim}}
\ {\color{black} \fbox{7}}\\
 &\qquad \to 
(q_1^{(3)},p_1^{(3)},c_1^{(3)} \ve^{-1},c_2^{(3)}\ve )  \mbox{{\color{black}: Returned}}.
\end{align*}
Since the mapping is symmetric with respect to $(q_1,p_1)\leftrightarrow (q_2,p_2)$,
there are the counterparts of these sequences.
The inclusion relations of these bases of blow-ups are the same with \eqref{inclusion}.

\section{Space of initial conditions and linearisation on the N\'eron-Severi lattices}

In this section we construct a space of initial conditions by blowing up the defining variety along singularities of the previous section. 
Recall that as a complex manifold, in local coordinates $U\subset \C^N$, blowing up along a subvariety $V$ of dimension $N-k$, $k\geq 2$, written as 
$$x_1-h_1(x_{k+1},\dots x_N)=\dots=x_k-h_k(x_{k+1},\dots x_N)=0,$$
where $h_i$'s are holomorphic functions,  
is a birational morphism $\pi:X\to U$ such that $X=\{U_i\}$ is an open variety given by
\begin{align*}
U_i=\{(u_1^{(i)},\dots, u_k^{(i)},x_{k+1},\dots x_N )\in \C^N\}\quad (i=1,\dots,k)
\end{align*}
with $\pi: U_i \to U$: 
\begin{align*}(x_1,\dots,x_N)=&(u_1^{(i)}u_i^{(i)}+h_1, \dots, u_{i-1}^{(i)}u_i^{(i)}+h_{i-1},
u_i^{(i)}+h_i, \\
\quad &u_{i+1}^{(i)}u_i^{(i)}+h_{i+1}\dots, u_k^{(i)}u_i^{(i)}+h_k,x_{k+1},\dots,x_N).
\end{align*} 
It is convenient to write the coordinates of $U_i$ as
$$\left(\frac{x_1-h_1}{x_i-h_i},\dots,\frac{x_{i-1}-h_{i-1}}{x_i-h_i},x_i-h_i,
\frac{x_{i+1}-h_{i+1}}{x_i-h_i},\dots,\frac{x_k-h_k}{x_i-h_i},x_{k+1},\dots x_N \right).$$
The exceptional divisor $E$ is written as $u_i=0$ in $U_i$ and each point in the center of 
blowup corresponds to a subvariety isomorphic to $\P^{k-1}$:
$(x_1-h_1: \dots: x_{k-1}-h_k)$. Hence $E$ is locally a direct product $V \times \P^{k-1}$.
We called such $\P^{k-1}$ a fiber of the exceptional divisor. 
(In algebraic setting the affine charts often need to be embedded into higher dimensional space.)

\begin{theorem}\label{SIC}
Each mapping \eqref{Map1} or \eqref{Map2} can be lifted to a pseudo-automorphism on a rational projective variety $\mathcal{X}$ obtained by successive 
16 blow-ups from $(\P^1)^4$, where the center of each blow-up, $C_i$ ($i=1,\dots,16$), is two-dimensional  sub-variety.
\end{theorem}

Center $C_i$'s are given by the following data, where we only write one of the affine coordinate of the center variety or the exceptional divisor. The other coordinates can be obtained automatically (see Fig.~\ref{figure_blowup} and Fig.~\ref{figure_blowup2}). \\

\noindent {\it Case $A_2^{(1)}+A_2^{(1)}$}:\\
The center $C_i$ of $i$-the blowing up and one of the new coordinate systems $U_i$ obtained by the blowing-up are
\begin{align*}
\begin{array}{ll}
C_1: q_1^{-1}=p_1^{-1}=0  &U_1: (u_1,v_1, q_2, p_2)=(q_1^{-1}, q_1p_1^{-1}, q_2, p_2)\\
C_2: u_1=v_1+1=0&U_2: (u_2,v_2, q_2, p_2)=(u_1,u_1^{-1}(v_1+1), q_2, p_2)\\
C_3: u_2=v_2+b^{(1)}=0&U_3: (u_3,v_3, q_2, p_2)=(u_2,u_2^{-1} (v_2+b^{(1)}), q_2, p_2)\\
C_4: u_3=v_3+(b^{(1)})^2+a_0^{(1)}=0\hspace{-1cm}& \\
&\hspace{-2cm} U_4: (u_4,v_4, q_2, p_2)=(u_3,u_3^{-1} (v_3+(b^{(1)})^2+a_0^{(1)}), q_2, p_2)\\
C_{5}: q_1^{-1}=p_1=0  &U_{5}:  (u_{5}, v_5, q_2 , p_2)=(q_1^{-1}, q_1p_1, q_2, p_2)\\
C_{6}: u_{5}=v_{5}-a_1^{(1)}=0  &U_{6}:  (u_{6},v_6, q_2, p_2)=(u_{5},u_{5}^{-1}(v_{5}-a_1^{(1)}), q_2, p_2)\\
C_{7}: q_1=p_1^{-1}=0  &U_{7} (v_{7}, u_7, q_2 , p_2)=(q_1p_1, p_1^{-1}, q_2, p_2)\\
C_{8}: u_{7}=v_{7}+a_2^{(1)}=0  &U_{8}:  (v_{8},u_8, q_2, p_2)=(u_{7}^{-1}(u_{7}+a_2^{(1)}),u_{7}, q_2, p_2)
\\%\end{array}
%\end{align*}
%\begin{align*}
%\begin{array}{ll}
C_9: p_2^{-1}=q_2^{-1}=0  &U_9:  (q_1, p_1,u_9,v_9)=(q_1, p_1 , q_2^{-1}, p_2^{-1}q_2)\\
C_{10}: u_9=v_9+1=0&U_{10}: (q_1, p_1,u_{10},v_{10})=(q_1, p_1, u_9,u_9^{-1}(v_9+1))\\
C_{11}: u_{10}=v_{10}+b^{(2)}=0&U_{11}: (q_1, p_1,u_{11},v_{11})=(q_1, p_1, u_{10},u_{10}^{-1}(v_{10}+b^{(2)}) )\\
C_{12}: u_{11}=v_{11}+(b^{(2)})^2+a_0^{(2)}=0\hspace{-1cm}&\\
&\hspace{-2cm}U_{12}: (q_1, p_1,u_{12},v_{12})=(q_1, p_1, u_{11},u_{11}^{-1}(v_{11}+(b^{(2)})^2+a_0^{(2)}))\\
C_{13}: p_2=q_2^{-1}=0  &U_{13}:  (q_1, p_1,u_{13},v_{13})=(q_1, p_1, q_2^{-1}, p_2q_2)\\
C_{14}: u_{13}=v_{13}-a_1^{(2)}=0  &U_{14}:  (q_1, p_1,u_{14},v_{14})=(q_1, p_1, u_{13},u_{13}^{-1}(v_{13}-a_1^{(2)}))\\
C_{15}: p_2^{-1}=q_2=0  &U_{15}:  (q_1, p_1, v_{15},u_{15})=(q_1, p_1, p_2q_2, p_2^{-1})\\
C_{16}: u_{15}=v_{15}+a_2^{(2)}=0  &U_{16}:  (q_1, p_1, v_{16}, u_{16})=(q_1, p_1, u_{15}^{-1}(v_{15}+a_2^{(2)}), u_{15})
\end{array}
\end{align*}
with $a_0^{(j)}=0$, $a_1^{(j)}=-a_2^{(j)}=a$ and $b^{(j)}=b$ for $j=1,2$,
where parameters $a={a_i^{(j)}, b^{(j)}}$ are introduced for ``deautonomisation'' as explained in the next section.\\

\noindent {\it Case $A_5^{(1)}$}:
%The center $C_i$ of $i$-the blowing up and one of the new coordinate systems $U_i$ obtained by the blowing-up are
\begin{align*}
\begin{array}{ll}
C_1: q_2^{-1}=p_1^{-1}=0  &U_1: (q_1,v_1,u_1,p_2)=(q_1,q_2p_1^{-1},q_2^{-1},p_2)\\
C_2: u_1=v_1+1=0&U_2: (q_1,v_2,u_2,p_2)=(q_1,u_1^{-1}(v_1+1),u_1,p_2)\\
C_3: u_2=q_1+p_2-b_1=0&U_3: (q_1,v_2,u_3,v_3)=(q_1,v_2,u_2,u_2^{-1}(q_1+p_2-b_1)\\
C_4: u_3=v_3+a_0=0&U_4: (q_1,v_2,u_4,v_4)=(q_1,v_2,u_3, u_3^{-1}(v_3+a_0))\\
C_5: q_2^{-1}=p_2=0  &U_5:  (q_1,p_1,u_5,v_5)=(q_1,p_1,q_2^{-1}, p_2q_2)\\
C_6: u_5=v_5+a_2=0  &U_6:  (q_1,p_1,u_5,v_5)=(q_1,p_1,u_5,u_5^{-1}(v_5+a_2))\\
C_7: q_1=p_1^{-1}=0  &U_7:  (v_7,u_7,q_2,p_2)=(q_1p_1,p_1^{-1},q_2, p_2)\\
C_8: u_7=v_7-a_4=0  &U_8:  (v_8,u_8,q_2,p_2)=(u_7^{-1}(v_7-a_4),u_7,q_2,p_2)
\\%\end{array}
%\end{align*}
%\begin{align*}
%\begin{array}{ll}
C_9: q_1^{-1}=p_2^{-1}=0  &U_9:  (u_9,p_1,q_2,v_9)=(q_1^{-1},p_1,q_2,q_1p_2^{-1})\\
C_{10}: u_9=v_9+1=0&U_{10}: (u_{10},p_1,q_2,v_{10})=(u_9,p_1,q_2,u_9^{-1}(v_9+1))\\
C_{11}: u_{10}=q_2+p_1-b_2=0&U_{11}: (u_{11},v_{11},q_2,v_{10})=(u_{10},u_{10}^{-1}(q_2+p_1-b_2),q_2,v_{10})\\
C_{12}: u_{11}=v_{11}+a_3=0&U_{12}: (u_{12},v_{12},q_2,v_{10})=(u_{11},u_{11}^{-1}(v_{11}+a_3), q_2,v_{10})\\
C_{13}: q_1^{-1}=p_1=0  &U_{13}:  (u_{13},v_{13},q_2,p_2)=(q_1^{-1},q_1p_1,q_2, p_2)\\
C_{14}: u_{13}=v_{13}+a_5=0  &U_{14}:  (u_{14},v_{14},q_2,p_2)=(u_{13},p_2,q_2,u_{13}^{-1}(v_{13}+a_5))\\
C_{15}: p_2^{-1}=q_2=0  &U_{15}: (q_1,p_1,v_{15},u_{15})=(q_1,p_1,p_2q_2,p_2^{-1})\\
C_{16}: u_{15}=v_{15}-a_1=0  &U_{16}:  (q_1,p_1,v_{16},u_{16})=(q_1,p_1,u_{15}^{-1}(v_{15}-a_1),u_{15})
\end{array}
\end{align*}
with 
$$a_0=a_3=0,\quad a_1=a_4=a,\quad a_2=a_5=-a.$$

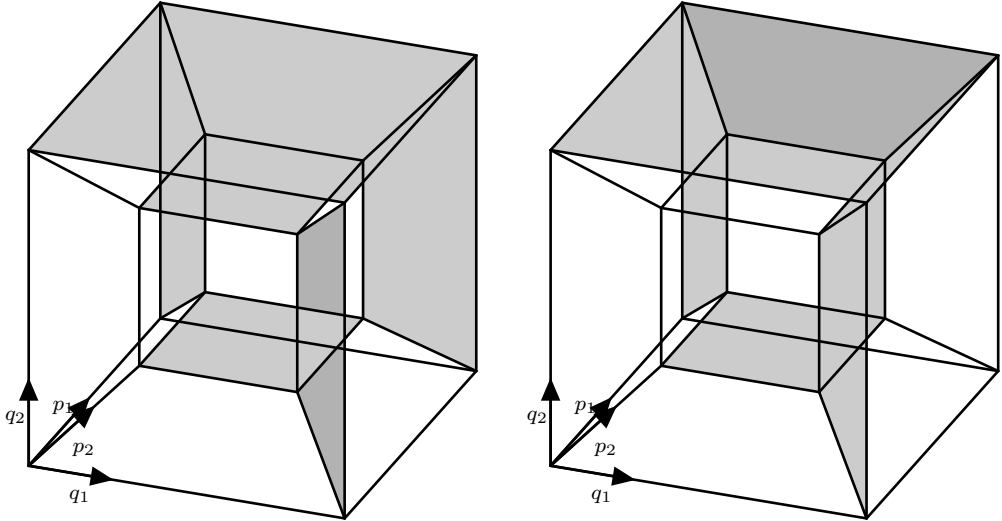
\begin{figure}[htbp]
\begin{center}
\begin{tikzpicture}[line cap=round,line join=round,>=triangle 45,x=0.7cm,y=0.7cm]
\clip(-0.6,-4) rectangle (9.12,7.44);
\fill[color=light-gray] (8.5,5.3) -- (6.35,3.3) -- (6.35,0.3) -- (8.5,-0.7)   -- cycle;
\fill[color=light-gray2]  (5.1,1.9) -- (6.,2.5) -- (6.,-3.5) -- (5.1,-1.1)   -- cycle;
\fill[color=light-gray]  (3.35,3.8) -- (3.35,0.8) -- (2.5,0.3) -- (2.5,6.3)   -- cycle;
\fill[color=light-gray] (2.5,6.3) -- (0.,3.5) -- (6.,2.5) -- (8.5,5.3)    -- cycle;
\fill[color=light-gray]  (3.35,3.8) -- (2.1,2.4) -- (5.1,1.9) -- (6.35,3.3) -- cycle;
\fill[color=light-gray] (3.35,0.8) -- (2.1,-0.6) -- (5.1,-1.1) -- (6.35,0.3)  -- cycle;
\draw [line width=1.pt] (0.,3.5)-- (0.,-2.5);
\draw [line width=1.pt] (0.,-2.5)-- (6.,-3.5);
\draw [line width=1.pt] (0.,-2.5)-- (2.5,0.3);
\draw [line width=1.pt] (6.,-3.5)-- (8.5,-0.7);
\draw [line width=1.pt] (8.5,-0.7)-- (2.5,0.3);
\draw [line width=1.pt] (2.5,0.3)-- (2.5,6.3);
\draw [line width=1.pt] (2.5,6.3)-- (0.,3.5);
\draw [line width=1.pt] (2.5,6.3)-- (8.5,5.3);
\draw [line width=1.pt] (8.5,5.3)-- (8.5,-0.7);
\draw [line width=1.pt] (8.5,5.3)-- (6.,2.5);
\draw [line width=1.pt] (6.,2.5)-- (0.,3.5);
\draw [line width=1.pt] (6.,2.5)-- (6.,-3.5);
\draw [line width=1.pt] (2.1,2.4)-- (2.1,-0.6);
\draw [line width=1.pt] (2.1,-0.6)-- (5.1,-1.1);
\draw [line width=1.pt] (2.1,-0.6)-- (3.35,0.8);
\draw [line width=1.pt] (5.1,-1.1)-- (6.35,0.3);
\draw [line width=1.pt] (6.35,0.3)-- (3.35,0.8);
\draw [line width=1.pt] (3.35,0.8)-- (3.35,3.8);
\draw [line width=1.pt] (3.35,3.8)-- (2.1,2.4);
\draw [line width=1.pt] (3.35,3.8)-- (6.35,3.3);
\draw [line width=1.pt] (6.35,3.3)-- (6.35,0.3);
\draw [line width=1.pt] (6.35,3.3)-- (5.1,1.9);
\draw [line width=1.pt] (5.1,1.9)-- (2.1,2.4);
\draw [line width=1.pt] (5.1,1.9)-- (5.1,-1.1);
\draw [line width=1.pt] (2.1,-0.6)-- (0.,-2.5);
\draw [line width=1.pt] (5.1,-1.1)-- (6.,-3.5);
\draw [line width=1.pt] (6.35,0.3)-- (8.5,-0.7);
\draw [line width=1.pt] (3.35,0.8)-- (2.5,0.3);
\draw [line width=1.pt] (2.1,2.4)-- (0.,3.5);
\draw [line width=1.pt] (5.1,1.9)-- (6.,2.5);
\draw [line width=1.pt] (6.35,3.3)-- (8.5,5.3);
\draw [line width=1.pt] (3.35,3.8)-- (2.5,6.3);
\draw [->,line width=1.pt] (0.,-2.5) -- (1.611891891891892,-2.768648648648649);
\draw [->,line width=1.pt] (0.,-2.5) -- (0.,-0.8274616682650864);
\draw [->,line width=1.pt] (0.,-2.5) -- (1.1648126942581385,-1.1954097824308851);
\draw [->,line width=1.pt] (0.,-2.5) -- (1.2588972156884197,-1.360997757234287);
\begin{scriptsize}
\draw[color=black] (0.97,-3.08) node {$q_1$};
\draw[color=black] (-0.25,-1.58) node {$q_2$};
\draw[color=black] (0.67,-1.4) node {$p_1$};
\draw[color=black] (1.05,-2.2) node {$p_2$};
\end{scriptsize}
\end{tikzpicture}
\begin{tikzpicture}[line cap=round,line join=round,>=triangle 45,x=0.7cm,y=0.7cm]
\clip(-0.6,-4) rectangle (9.12,7.44);
\fill[color=light-gray](0.,3.5) -- (2.5,6.3) -- (8.5,5.3) -- (6.,2.5)   -- cycle;
\fill[color=light-gray2] (2.5,6.3) -- (3.35,3.8) -- (6.35,3.3) -- (8.5,5.3)  -- cycle;
\fill[color=light-gray] (2.5,6.3) -- (3.35,3.8) -- (3.35,0.8) -- (2.5,0.3)   -- cycle;
\fill[color=light-gray]  (6.,2.5) -- (5.1,1.9) -- (5.1,-1.1) -- (6.,-3.5)   -- cycle;
\fill[color=light-gray]  (6.35,3.3) -- (5.1,1.9) -- (5.1,-1.1) -- (6.35,0.3) -- cycle;
\fill[color=light-gray]  (3.35,0.8) -- (2.1,-0.6) -- (5.1,-1.1) -- (6.35,0.3)  -- cycle;
\draw [line width=1.pt] (0.,3.5)-- (0.,-2.5);
\draw [line width=1.pt] (0.,-2.5)-- (6.,-3.5);
\draw [line width=1.pt] (0.,-2.5)-- (2.5,0.3);
\draw [line width=1.pt] (6.,-3.5)-- (8.5,-0.7);
\draw [line width=1.pt] (8.5,-0.7)-- (2.5,0.3);
\draw [line width=1.pt] (2.5,0.3)-- (2.5,6.3);
\draw [line width=1.pt] (2.5,6.3)-- (0.,3.5);
\draw [line width=1.pt] (2.5,6.3)-- (8.5,5.3);
\draw [line width=1.pt] (8.5,5.3)-- (8.5,-0.7);
\draw [line width=1.pt] (8.5,5.3)-- (6.,2.5);
\draw [line width=1.pt] (6.,2.5)-- (0.,3.5);
\draw [line width=1.pt] (6.,2.5)-- (6.,-3.5);
\draw [line width=1.pt] (2.1,2.4)-- (2.1,-0.6);
\draw [line width=1.pt] (2.1,-0.6)-- (5.1,-1.1);
\draw [line width=1.pt] (2.1,-0.6)-- (3.35,0.8);
\draw [line width=1.pt] (5.1,-1.1)-- (6.35,0.3);
\draw [line width=1.pt] (6.35,0.3)-- (3.35,0.8);
\draw [line width=1.pt] (3.35,0.8)-- (3.35,3.8);
\draw [line width=1.pt] (3.35,3.8)-- (2.1,2.4);
\draw [line width=1.pt] (3.35,3.8)-- (6.35,3.3);
\draw [line width=1.pt] (6.35,3.3)-- (6.35,0.3);
\draw [line width=1.pt] (6.35,3.3)-- (5.1,1.9);
\draw [line width=1.pt] (5.1,1.9)-- (2.1,2.4);
\draw [line width=1.pt] (5.1,1.9)-- (5.1,-1.1);
\draw [line width=1.pt] (2.1,-0.6)-- (0.,-2.5);
\draw [line width=1.pt] (5.1,-1.1)-- (6.,-3.5);
\draw [line width=1.pt] (6.35,0.3)-- (8.5,-0.7);
\draw [line width=1.pt] (3.35,0.8)-- (2.5,0.3);
\draw [line width=1.pt] (2.1,2.4)-- (0.,3.5);
\draw [line width=1.pt] (5.1,1.9)-- (6.,2.5);
\draw [line width=1.pt] (6.35,3.3)-- (8.5,5.3);
\draw [line width=1.pt] (3.35,3.8)-- (2.5,6.3);
\draw [->,line width=1.pt] (0.,-2.5) -- (1.611891891891892,-2.768648648648649);
\draw [->,line width=1.pt] (0.,-2.5) -- (0.,-0.8274616682650864);
\draw [->,line width=1.pt] (0.,-2.5) -- (1.1648126942581385,-1.1954097824308851);
\draw [->,line width=1.pt] (0.,-2.5) -- (1.2588972156884197,-1.360997757234287);
\begin{scriptsize}
\draw[color=black] (0.97,-3.08) node {$q_1$};
\draw[color=black] (-0.25,-1.58) node {$q_2$};
\draw[color=black] (0.67,-1.4) node {$p_1$};
\draw[color=black] (1.05,-2.2) node {$p_2$};
\end{scriptsize}
\end{tikzpicture}
\caption{left: Case $A_2^{(1)}+A_2^{(1)}$, right: Case $A_5^{(1)}$, 
gray parallelograms: the centers $C_1$, $C_5$, $C_7$, $C_9$, $C_{13}$, $C_{15}$ for both cases}
    \label{figure_blowup}
\end{center}
\end{figure}

\begin{figure}[htbp]
\begin{center}
\begin{tikzpicture}[line cap=round,line join=round,>=triangle 45,x=0.6cm,y=0.6cm]
\clip(-3.5,-10.5) rectangle (18.5,5.5);
\fill[color=light-gray] (-1.0,3.0) -- (-1.0,-0.5) -- (1.5,1.5) -- (1.5,5.0) -- cycle;
\draw [line width=1.pt] (-1.0,-0.5)-- (-1.0,3.0);
\draw [line width=1.pt] (-1.0,-0.5)-- (3.5,-0.5);
\draw [line width=1.pt] (-1.0,3.0)-- (3.5,3.0);
\draw [line width=1.pt] (3.5,3.0)-- (3.5,-0.5);
\draw [line width=1.pt] (-1.0,-0.5)-- (1.5,1.5);
\draw [line width=1.pt] (-1.0,-0.5)-- (-3.0,1.5);
\draw [line width=1.pt] (-1.0,3.0)-- (1.5,5.0);
\draw [line width=1.pt] (3.5,-0.5)-- (6.0,1.5);
\draw [line width=1.pt] (6.0,1.5)-- (1.5,1.5);
\draw [line width=1.pt] (1.5,1.5)-- (1.5,5.0);
\draw [line width=1.pt] (1.5,5.0)-- (6.0,5.0);
\draw [line width=1.pt] (6.0,5.0)-- (6.0,1.5);
\draw [line width=1.pt] (3.5,3.0)-- (6.0,5.0);
\draw [->,line width=1.pt] (-1.0,-0.5) -- (-2.0,0.5);
\draw [->,line width=1.pt] (-1.0,-0.5) -- (0.75,-0.5);
\draw [->,line width=1.pt] (-1.0,-0.5) -- (0.0,-0.5+0.8);
\draw [->,line width=1.pt] (-1.0,-0.5) -- (-1.0,1.1);
\begin{scriptsize}
\draw[color=black] (-2.0,-0.5) node {$q_2^{-1}$};
\draw[color=black] (0.6,-1.1) node {$p_1^{-1}$};
\draw[color=black] (0.37,0.0) node {$q_1$};
\draw[color=black] (-1.4,1.0) node {$p_2$};
\end{scriptsize}
%\end{tikzpicture}
%
%\begin{tikzpicture}[line cap=round,line join=round,>=triangle 45,x=0.6cm,y=0.6cm]
%\clip(-5.5,-2.5) rectangle (6.5,5.5);
\fill[color=light-gray] (1.0+12,3.0) -- (1.0+12,-0.5) -- (3.5+12,1.5) -- (3.5+12,5.0) -- cycle;
\draw [line width=1.pt] (-1.0+12,-0.5)-- (-1.0+12,3.0);
\draw [line width=1.pt] (-1.0+12,-0.5)-- (3.5+12,-0.5);
\draw [line width=1.pt] (-1.0+12,3.0)-- (3.5+12,3.0);
\draw [line width=1.pt] (3.5+12,3.0)-- (3.5+12,-0.5);
\draw [line width=1.pt] (-1.0+12,-0.5)-- (1.5+12,1.5);
\draw [line width=1.pt] (-1.0+12,-0.5)-- (-3.0+12,1.5);
\draw [line width=1.pt] (-1.0+12,3.0)-- (1.5+12,5.0);
\draw [line width=1.pt] (3.5+12,-0.5)-- (6.0+12,1.5);
\draw [line width=1.pt] (6.0+12,1.5)-- (1.5+12,1.5);
\draw [line width=1.pt] (1.5+12,1.5)-- (1.5+12,5.0);
\draw [line width=1.pt] (1.5+12,5.0)-- (6.0+12,5.0);
\draw [line width=1.pt] (6.0+12,5.0)-- (6.0+12,1.5);
\draw [line width=1.pt] (3.5+12,3.0)-- (6.0+12,5.0);
\draw [->,line width=1.pt] (-1.0+12,-0.5) -- (-2.0+12,0.5);
\draw [->,line width=1.pt] (-1.0+12,-0.5) -- (0.75+12,-0.5);
\draw [->,line width=1.pt] (-1.0+12,-0.5) -- (0.0+12,-0.5+0.8);
\draw [->,line width=1.pt] (-1.0+12,-0.5) -- (-1.0+12,1.1);
\draw [->,line width=2.pt] (-4.0+12,1.5) -- (-5+12,1.5);
\begin{scriptsize}
\draw[color=black] (-2.0+12,-0.5) node {$u_1$};
\draw[color=black] (0.6+12,-1.1) node {$v_1$};
\draw[color=black] (0.37+12,0.0) node {$q_1$};
\draw[color=black] (-1.4+12,1.0) node {$p_2$};
\draw[color=black] (-4.0+12,0.5) node {blowdown};
\end{scriptsize}
\draw[color=black] (5.5+12,-0.5) node {{\large $E_1$}};
%\end{tikzpicture}
%\begin{tikzpicture}[line cap=round,line join=round,>=triangle 45,x=0.6cm,y=0.6cm]
%\clip(-5.5,-2.5) rectangle (6.5,5.5);
\fill[color=light-gray] (1.5,5.0-8) -- (-1.0,0.5-8) -- (3.5,0.5-8) -- (6.0,5.0-8) -- cycle;
\draw [line width=1.pt] (-1.0,-0.5-8)-- (-1.0,3.0-8);
\draw [line width=1.pt] (-1.0,-0.5-8)-- (3.5,-0.5-8);
\draw [line width=1.pt] (-1.0,3.0-8)-- (3.5,3.0-8);
\draw [line width=1.pt] (3.5,3.0-8)-- (3.5,-0.5-8);
\draw [line width=1.pt] (-1.0,-0.5-8)-- (1.5,1.5-8);
\draw [line width=1.pt] (-1.0,-0.5-8)-- (-3.0,1.5-8);
\draw [line width=1.pt] (-1.0,3.0-8)-- (1.5,5.0-8);
\draw [line width=1.pt] (3.5,-0.5-8)-- (6.0,1.5-8);
\draw [line width=1.pt] (6.0,1.5-8)-- (1.5,1.5-8);
\draw [line width=1.pt] (1.5,1.5-8)-- (1.5,5.0-8);
\draw [line width=1.pt] (1.5,5.0-8)-- (6.0,5.0-8);
\draw [line width=1.pt] (6.0,5.0-8)-- (6.0,1.5-8);
\draw [line width=1.pt] (3.5,3.0-8)-- (6.0,5.0-8);
\draw [->,line width=1.pt] (-1.0,-0.5-8) -- (-2.0,0.5-8);
\draw [->,line width=1.pt] (-1.0,-0.5-8) -- (0.75,-0.5-8);
\draw [->,line width=1.pt] (-1.0,-0.5-8) -- (0.0,-0.5+0.8-8);
\draw [->,line width=1.pt] (-1.0,-0.5-8) -- (-1.0,1.1-8);
\draw [->,line width=2.pt] (7.0,-2.5) -- (8.0,-1.7);
\begin{scriptsize}
\draw[color=black] (-2.0,-0.5-8) node {$u_2$};
\draw[color=black] (0.6,-1.1-8) node {$v_2$};
\draw[color=black] (0.37,0.0-8) node {$q_1$};
\draw[color=black] (-1.4,1.0-8) node {$p_2$};
\draw[color=black] (9.0, -2.2) node {blowdown};
\end{scriptsize}
\draw[color=black] (5.5,-0.5-8) node {{\large $E_2$}};
%\end{tikzpicture}
%
%\begin{tikzpicture}[line cap=round,line join=round,>=triangle 45,x=0.6cm,y=0.6cm]
%\clip(-5.5,-2.5) rectangle (6.5,5.5);
\fill[color=light-gray] (1.5+12,3.3-8) -- (-1.0+12,1.3-8) -- (3.5+12,1.3-8) -- (6.0+12,3.3-8) -- cycle;
\draw [line width=1.pt] (-1.0+12,-0.5-8)-- (-1.0+12,3.0-8);
\draw [line width=1.pt] (-1.0+12,-0.5-8)-- (3.5+12,-0.5-8);
\draw [line width=1.pt] (-1.0+12,3.0-8)-- (3.5+12,3.0-8);
\draw [line width=1.pt] (3.5+12,3.0-8)-- (3.5+12,-0.5-8);
\draw [line width=1.pt] (-1.0+12,-0.5-8)-- (1.5+12,1.5-8);
\draw [line width=1.pt] (-1.0+12,-0.5-8)-- (-3.0+12,1.5-8);
\draw [line width=1.pt] (-1.0+12,3.0-8)-- (1.5+12,5.0-8);
\draw [line width=1.pt] (3.5+12,-0.5-8)-- (6.0+12,1.5-8);
\draw [line width=1.pt] (6.0+12,1.5-8)-- (1.5+12,1.5-8);
\draw [line width=1.pt] (1.5+12,1.5-8)-- (1.5+12,5.0-8);
\draw [line width=1.pt] (1.5+12,5.0-8)-- (6.0+12,5.0-8);
\draw [line width=1.pt] (6.0+12,5.0-8)-- (6.0+12,1.5-8);
\draw [line width=1.pt] (3.5+12,3.0-8)-- (6.0+12,5.0-8);
\draw [->,line width=1.pt] (-1.0+12,-0.5-8) -- (-2.0+12,0.5-8);
\draw [->,line width=1.pt] (-1.0+12,-0.5-8) -- (0.75+12,-0.5-8);
\draw [->,line width=1.pt] (-1.0+12,-0.5-8) -- (0.0+12,-0.5+0.8-8);
\draw [->,line width=1.pt] (-1.0+12,-0.5-8) -- (-1.0+12,1.1-8);
\draw [->,line width=2.pt] (-4+12,1.5-8) -- (-5+12,1.5-8);
\begin{scriptsize}
\draw[color=black] (-2.0+12,-0.5-8) node {$u_3$};
\draw[color=black] (0.6+12,-1.1-8) node {$v_2$};
\draw[color=black] (0.37+12,0.0-8) node {$q_1$};
\draw[color=black] (-1.4+12,1.0-8) node {$v_3$};
\draw[color=black] (-4+12,0.5-8) node {blowdown};
\end{scriptsize}
\draw[color=black] (5.5+12,-0.5-8) node {{\large $E_3$}};
\end{tikzpicture}
\caption{Case $A_5^{(1)}$, gray parallelograms: the centers $C_1$, $C_2$, $C_3$, $C_4$, rectangulars: the exceptional divisors $E_1$, $E_2$, $E_3$}
    \label{figure_blowup2}
\end{center}
\end{figure}
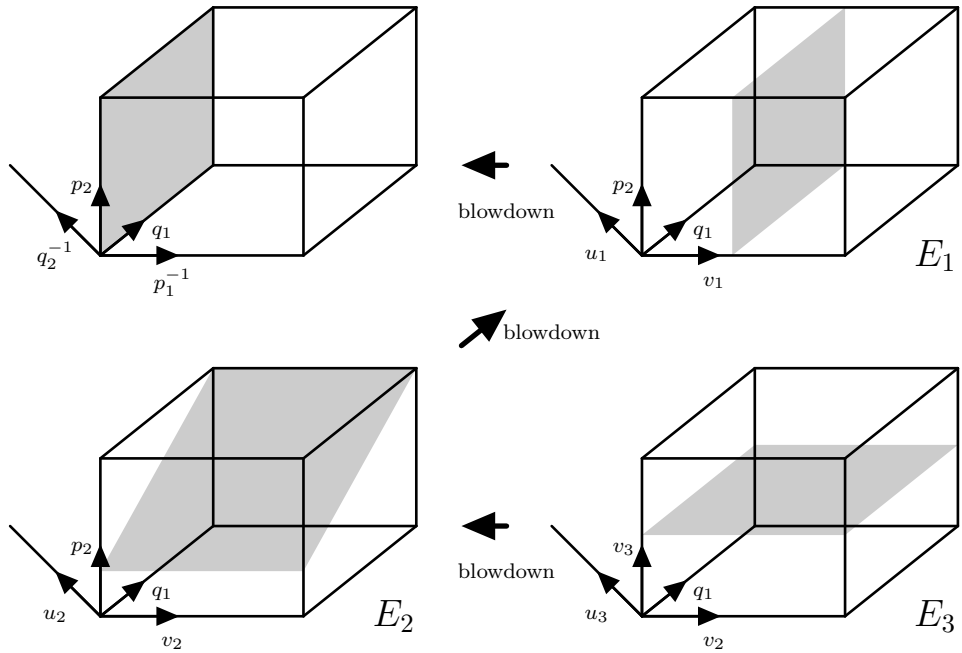

\begin{remark} \label{int_centers}
Some centers (e.g., $C_1$ and $C_9$) intersect with each other but do not have  inclusion relation. In this case, the variety depends on the order of blowups. However, since generic points are not in the intersection points, the varieties are pseudo-isomorphic with each other.   
\end{remark}

In both cases, the inclusion relations of total transforms of exceptional divisors $E_i$'s are the same with \eqref{inclusion} as
\begin{align}
&E_1\supset E_2 \supset E_3 \supset E_4, 
 &&E_5 \supset E_6,  &&E_7 \supset E_8, \nonumber\\  
&E_9\supset E_{10} \supset E_{11} \supset E_{12}, 
 &&E_{13} \supset E_{14},  &&E_{15} \supset E_{16}.
 \label{inclusion2}
\end{align}

\begin{proof}
The proof of the theorem is long but straightforward. We omit the detail, but 
we can show that any divisors in $\cX$ are mapped to divisors in $\cX$.
For example, in $A_5^{(1)}$ case, the exceptional divisor $E_4$ is described as $u_4=0$ in $U_4$, while $E_8$ as $u_8=0$ in $U_8$.
The mapping $\varphi$ from $U_4$ to $U_8$ under $a_0=0$ and $a_4=a$ is
\begin{align*}
(\bar{v}_8,\bar{u}_8,\bar{q}_2,\bar{p}_2)=&\left(-v_4, u_4, a q_1^{-1} + (b_2 + v_2 - b_2 u_4 v_2)(1 - u_4 v_2)^{-1}, q_1\right)
\end{align*} 
and hence $u_4=0$ implies $\bar{u}_8$ in generic (i.e. $q_1\neq 0$).
\end{proof}

Similarly computation to this proof yields the following theorem.

\begin{theorem}
The push-forward action of $\varphi$ on $H^2(\mathcal{X}, \Z)$ is as follows:\\
Case $A_2^{(1)}+A_2^{(1)}$:
\begin{align}\label{actionA2A2}
\begin{array}{l}
H_{q_1}\mapsto H_{p_2}, \quad
H_{p_1}\mapsto H_{q_2}+2H_{p_2}-E_{9,10,13,14}\\
H_{q_2}\mapsto H_{p_1}, \quad
H_{p_2}\mapsto H_{q_1}+2H_{p_1}-E_{1,2,5,6}\\
E_{1}\mapsto H_{p_2}-E_{10}, \quad 
E_{2}\mapsto H_{p_2}-E_{9}, \quad
E_{3}\mapsto E_{15}, \quad
E_{4}\mapsto E_{16}, \\
E_{5}\mapsto E_{11}, \quad
E_{6}\mapsto E_{12}, \quad
E_{7}\mapsto H_{p_2}-E_{14}, \quad 
E_{8}\mapsto H_{p_2}-E_{13}, \\
E_{9}\mapsto H_{p_1}-E_{2}, \quad
E_{10}\mapsto H_{p_1}-E_{1}, \quad 
E_{11}\mapsto E_{7}, \quad 
E_{12}\mapsto E_{8}, \\
E_{13}\mapsto E_{3}, \quad 
E_{14}\mapsto E_{4}, \quad
E_{15}\mapsto H_{p_1}-E_{6}, \quad
E_{16}\mapsto H_{p_1}-E_{5}
\end{array}
\end{align}

\noindent Case $A_5^{(1)}$:
\begin{align}\label{actionA5}
\begin{array}{l}
H_{q_1}\mapsto H_{p_2},\quad
H_{p_1}\mapsto H_{p_1}+H_{q_2}+H_{p_2}-E_{1,2,5,6}\\
H_{q_2}\mapsto H_{p_1},\quad
H_{p_2}\mapsto H_{q_1}+H_{p_1}+H_{p_2}-E_{9,10,13,14}\\
E_1\mapsto H_{p_1}-E_2,\quad 
E_2\mapsto H_{p_1}-E_1,\quad
E_3\mapsto E_7,\quad
E_4\mapsto E_8,\\
E_5\mapsto E_3,\quad 
E_6\mapsto E_4,\quad
E_7\mapsto H_{p_2}-E_6,\quad
E_8\mapsto H_{p_2}-E_5,\\
E_9\mapsto H_{p_2}-E_{10},\quad
E_{10}\mapsto H_{p_2}-E_9,\quad 
E_{11}\mapsto E_{15},\quad 
E_{12}\mapsto E_{16},\\
E_{13}\mapsto E_{11},\quad
E_{14}\mapsto E_{12},\quad
E_{15}\mapsto H_{p_1}-E_{14},\quad 
E_{16}\mapsto H_{p_1}-E_{13}
\end{array}
\end{align}
and the action on $H_2(\mathcal{X}, \Z)$ is given by \eqref{H2H2} with
\begin{align*}
J=\begin{bmatrix}
I_4&0\\0& -I_{16}
\end{bmatrix}.
\end{align*}
\end{theorem}

The actions \eqref{actionA2A2} and \eqref{actionA5} correspond to
singularity patterns in the previous section.
The pull-back actions are given by their inverse. 

\begin{corollary}
Both the degrees of mappings \eqref{Map1} and \eqref{Map2} grow quadratically.  
\end{corollary}

\begin{proof}
As mentioned in Section 2, the degrees are given by the coefficients of $H_i$'s  of $(\varphi^*)^n$, while 
the Jordan blocks of $\varphi^*$ consist of 
$$ \begin{bmatrix}1&1&0\\0&1&1\\0&0&1 \end{bmatrix} $$
and seventeen $1\times 1$ matrices whose absolute value is 1. 
\end{proof}

\begin{theorem}
 For Case $A_2^{(1)}+A_2^{(1)}$,
the linear system of the anticanonical divisor class $\delta=2 \sum_{i=1}^2(H_{q_i}+H_{p_i}) - \sum_{i=1}^{16} E_i$ is given by 
\begin{align}
(\alpha_0 +\alpha_1 I_1)(\beta_0+\beta_1 I_2)=&0
\end{align}
for any  $(\alpha_0:\alpha_1), (\beta_0:\beta_1) \in \P^1$, where $I_i$ are given by \eqref{I12} and fibers $\alpha_0+ \alpha_1I_1=0$ and $\alpha_0+\alpha_1I_2=1$ 
are mapped to each other, while for Case $A_5^{(1)}$, 
the linear system is given by 
\begin{align}
\alpha_0 +\alpha_1 I_1+\alpha_2 I_2=&0,
\end{align}
for any  $(\alpha_0:\alpha_1:\alpha_2) \in \P^2$, where $I_i$ are given by \eqref{I12.2}
and each fiber is preserved. 
\end{theorem}

\begin{remark}
In both cases the divisor defined by the coefficients of the symplectic form 
coincides with the canonical divisor. Indeed, for Case $A_2^{(1)}+A_2^{(1)}$, 
the divisor class corresponding to $dq_i\wedge dp_i$  is 
\begin{align*}
i=1:&-2(H_{q_1}-E_{1,5})-2(H_{p_1}-E_{1,7})-3E_{1-2}-2 E_{2-3}-E_{3-4}\\
& -E_{5-6}-E_{7-8}=-2H_{q_1}-2H_{p_1}+ E_{1,\dots,8}, \\
i=2:&\quad  q_1\leftrightarrow q_2, p_1\leftrightarrow p_2, E_j \leftrightarrow E_{j+8} \ (j=1,\dots,8) \mbox{ in the above},
\end{align*}
where $E_{i-j}$ denotes $E_i-E_j$,
while for Case $A_5^{(1)}$, that is
\begin{align*}
i=1:&-2(H_{q_1}-E_{9,13})-2(H_{p_1}-E_{1,7})-E_{1-2}-E_{7-8}-2E_{9-10}-2E_{10-11}\\
&-E_{11-12}-E_{13-14}=-2H_{q_1}-2H_{p_1}+ E_{1,2,7,8,11,12,13,14}\quad (i=1), \\
i=2:&\quad  q_1\leftrightarrow q_2, p_1\leftrightarrow p_2, E_j \leftrightarrow E_{j+8} \ (j=1,\dots,8) \mbox{ in the above}.
\end{align*}
\end{remark}

Hence, for Case $A_2^{(1)}+A_2^{(1)}$, the coefficients of the volume form corresponds to a decomposition of the anti-canonical divisor
\begin{align}
-K_{\cX}=&
2(H_{q_1}-E_{1,5})+2(H_{p_1}-E_{1,7})+2(H_{q_2}-E_{9,13})+2(H_{p_2}-E_{9,15})\nonumber\\
&+3E_{1-2}+2 E_{2-3}+E_{3-4}+E_{5-6}+E_{7-8}\nonumber\\
&+3E_{9-10}+2E_{10-11}+E_{11-12}+E_{13-14}+E_{15-16},\label{KXdecomposition}
\end{align}
while for Case $A_5^{(1)}$ it is 
\begin{align}
-K_{\cX}=&(H_{q_1}\leftrightarrow H_{q_2} \mbox{ in \eqref{KXdecomposition}}). \label{KXdecomposition2}
\end{align}

The above decompositions is left fixed by the action of the mapping. For example, in the case $A_2^{(1)}+A_2^{(1)}$ if we set $-K_{\cX}=D_1+...+D_{14}$ where  
$D_1=E_{1-2}$, $D_2=E_{2-3}$, $D_3=E_{3-4}$,  
$D_4=2H_{q_1}-E_{1,5}$, $D_5=E_{5-6}$, $D_6=2H_{p_1}-E_{1,7}$, 
$D_7=E_{7-8}$, $D_8=E_{9-10}$, $D_9=E_{10-11}$, $D_{10}=E_{11-12}$,  
$D_{11}=H_{q_2}-E_{9,13}$, $D_{12}=E_{13-14}$, $D_{13}=H_{p_2}-E_{9,15}$, 
$D_{14}=E_{15-16}$, then we have
\begin{align*}
\varphi_*:&(D_1, D_2, ..., D_{14})\mapsto\\
&\quad
(D_8, D_{13}, D_{14}, D_{9}, D_{10}, D_{11}, D_{12}, D_{1}, D_{6}, D_{7}, D_2, D_3, D_4, D_5 ).
\end{align*}
The set $\{D_1, D_2, ..., D_{14}\}$ is important because its orthogonal complement 
gives the symmetry group of the variety.

\section{Symmetries and deautonomisation}

Let us fix the decomposition of the anti-canonical divisor as \eqref{KXdecomposition}
or \eqref{KXdecomposition2}.

\begin{definition} An automorphism $s$ of the N\'eron-Severi bilattice is called a Cremona isometry if the following three properties are satisfied:\\
(a) $s$ preserves the intersection form;\\
(b) $s$ leaves the decomposition of $-K_{\cX}$ fixed;\\
(c) $s$ leaves the semigroup of effective classes of divisors invariant.
\end{definition}

In general, if a birational mapping on $\C^N$ can be lifted to a seudo-automorphism on $\cX$,
its action on the resulting N\'eron-Severi bilattice is always a Cremona isometry.
In order to consider the inverse problem, i.e. from a Cremona isometry to a birational mapping, at least we need to allow the mapping to move the centers of blow-ups, 
but keeping one of the decomposition of the anti-canonical divisor $\sum_i m_i D_i$ ($m_i\geq 1$). 
Here, the birational mapping is   
lifted to an isomorphism from $\cX_{\bf a}$ to $\cX_{{\bf a}'}$, where suffix ${\bf a}$
denotes parameters fixing the centers of blowups. 
Note that $\sum_i m_i D_i$ is the unique anti-canonical divisor for generic ${\bf a}$,
but not unique for the original $\cX$ and the deautonomisation depends on the choice of them. 
Here, we fix one of anti-canonical divisors of $\cX$.
This situation is the same with two dimensional case.
See \cite{CDT2017, WRG2017} in details.

In this section we construct a group of Cremona isometries for the $A_2^{(1)}+A_2^{(1)}$ and the $A_5^{(1)}$ cases and 
realise them as groups of birational mappings. Note that we do not know a canonical way to find root basis in $H^2(\cX_{{\bf a}},\Z)$, and hence we can not detect whether there are 
Cremona isometries outside of those groups or not. 
However, those groups act on a $\Z^6$ lattice in $H_2(\cX_{{\bf a}},\Z)$ nontrivially, which is the largest dimensional lattice orthogonal to the elements of decomposition of the anti-canonical divisor.\\

\noindent Case $A_2^{(1)}+A_2^{(1)}$:\\
Let $\cX_A$ denote a family of the space of initial conditions constructed in the previous section as
$$\cX_A:=\{\cX_{{\bf a}} \mbox{ {\rm in \S4}} ~|~{\bf a}=(a_0^{(1)},a_1^{(1)},a_2^{(1)},a_0^{(2)},a_1^{(2)},a_2^{(2)};b^{(1)},b^{(2)})\in \C^8 \}.$$ 
Then, there is a natural isomorphism between $H^2(\cX_{\bf a},\Z)\times H_2(\cX_{\bf a},\Z) \simeq H^2(\cX,\Z)\times H_2(\cX,\Z)$ as abstract lattices.

Let us define root vectors $\alpha_i^{(j)}$ and co-root vectors  $\check{\alpha}_i^{(j)}$ ($i=0,1,2$, $j=1,2$) so that the latter is orthogonal to all $D_i$, $i=1,...,14$, as
\begin{align}
\begin{array}{lll}
\alpha_0^{(1)}=H_{q_1}+H_{p_1}-E_{1,2,3,4}, 
&\alpha_1^{(1)}=H_{p_1}-E_{5,6}, 
&\alpha_2^{(1)}=H_{q_1}-E_{7,8}, \\
\alpha_0^{(2)}=H_{p_2}+H_{q_2}-E_{9,10,11,12},
&\alpha_1^{(2)}=H_{p_2}-E_{13,14}, 
&\alpha_2^{(2)}=H_{q_2}-E_{15,16}
\end{array}
\end{align}
and
\begin{align}
\begin{array}{lll}
\check{\alpha}_0^{(1)}=h_{q_1}+h_{p_1}-e_{1,2,3,4}, 
&\check{\alpha}_1^{(1)}=h_{q_1}-e_{5,6},
&\check{\alpha}_2^{(1)}=h_{p_1}-e_{7,8},\\
\check{\alpha}_0^{(2)}=h_{q_2}+h_{p_2}-e_{9,10,11,12},
&\check{\alpha}_1^{(2)}=h_{q_2}-e_{13,14},
&\check{\alpha}_2^{(2)}=h_{p_2}-e_{15,16}
\end{array}.
\end{align}
Then, the pairing $\langle \alpha_i^{(j)}, \check{\alpha}_k^{(l)} \rangle$ induces 
two of the affine root system of type $A_2^{(1)}$
with the null vectors $\delta^{(1)}=2H_{q_1}+2H_{p_1}-E_{1,\dots,8}$ and 
$\delta^{(2)}=2H_{q_2}+2H_{p_2}-E_{9,\dots,16}$ and
the null co-root vectors $\check{\delta}^{(1)}=2h_{q_1}+2h_{p_1}-e_{1,\dots,8}$ and 
$\check{\delta}^{(2)}=2h_{q_2}+2h_{p_2}-e_{9,\dots,16}$.
The Cartan matrix and the Dynkin diagram are 
\begin{center}
$
\begin{bmatrix}
2&-1&-1&0&0&0\\
-1&2&-1&0&0&0\\
-1&-1&2&0&0&0\\
0&0&0&2&-1&-1\\
0&0&0&-1&2&-1\\
0&0&0&-1&-1&-2
\end{bmatrix}$
\raisebox{-17mm}{
\begin{tikzpicture}[line cap=round,line join=round,>=triangle 45,x=0.9cm,y=0.9cm]
\clip(-2,-2) rectangle (5.4,1.5);
\draw [line width=1.pt] (0.,0.732)-- (-1.,-1.);
\draw [line width=1.pt] (-1.,-1.)-- (1.,-1.);
\draw [line width=1.pt] (1.,-1.)-- (0.,0.732);
\draw [fill=black] (0.,0.732) circle (2.5pt);
\draw [fill=black] (-1.,-1.) circle (2.5pt);
\draw [fill=black] (1.,-1.) circle (2.5pt);
\draw [line width=1.pt] (0+3.,0.732)-- (-1+3.,-1.);
\draw [line width=1.pt] (-1.+3,-1.)-- (1.+3,-1.);
\draw [line width=1.pt] (1.+3,-1.)-- (0.+3,0.732);
\draw [fill=black] (0.+3,0.732) circle (2.5pt);
\draw [fill=black] (-1.+3,-1.) circle (2.5pt);
\draw [fill=black] (1.+3,-1.) circle (2.5pt);
\begin{small}
\draw [color=black] (0.,1.2) node {$\alpha_0^{(1)}$};
\draw [color=black] (-1.,-1.5) node {$\alpha_1^{(1)}$};
\draw [color=black] (1.,-1.5) node {$\alpha_2^{(1)}$};
\draw [color=black] (3.,1.2) node {$\alpha_0^{(2)}$};
\draw [color=black] (2.,-1.5) node {$\alpha_1^{(2)}$};
\draw [color=black] (4.,-1.5) node {$\alpha_2^{(2)}$};
\end{small}
\end{tikzpicture}}.
\end{center}
Let $\widetilde{W}(A_2^{(1)}+A_2^{(1)})$ denote the extended affine Weyl group $\operatorname{Aut}(A_2^{(1)}+A_2^{(1)})\ltimes  (W(A_2^{(1)}) \times W(A_2^{(1)}))$, where $\operatorname{Aut}(A_2^{(1)}+A_2^{(1)})$ is the group of automorphisms of Dynkin diagram. 

Since $\check{\alpha}_i^{(j)}$'s are orthogonal to the elements of the decomposition of the anti-canonical divisor.  
Thus, if we define the action of the simple reflection $w_{\alpha_i^{(j)}}$ on the N\'eron-Severi bilattice as usual as
\begin{align}\label{Rootaction}
w_{\alpha_i^{(j)}}(D)= D+\langle D, \check{\alpha}_i^{(j)}   \rangle \alpha_i^{(j)}, \quad 
w_{\alpha_i^{(j)}}(d)=d+\langle \alpha_i^{(j)}, d   \rangle \check{\alpha}_i^{(j)}
\end{align}
for $D\in H^2(\cX_{{\bf a}}, \Z)$ and $d\in H_2(\cX_{{\bf a}}, \Z)$, it satisfies Condition (a) and (b) for Cremona isometries (Condition (c) is verified by realising as a birational mapping).
Moreover, the group of Dynkin automorphisms is generated by 
\begin{align*}
\sigma_{01}^{(1)}: &
  \alpha_0^{(1)} \leftrightarrow \alpha_1^{(1)}, \quad
  \check{\alpha}_0^{(1)} \leftrightarrow \check{\alpha}_1^{(1)}, \\
  &H_{p_1} \leftrightarrow H_{q_1}+H_{p_1}-E_{1}-E_{2}, \quad 
  E_{1} \leftrightarrow H_{q_1}-E_{2}, \\
  &E_{2} \leftrightarrow H_{q_1}-E_{1}, \quad
  E_{3} \leftrightarrow E_{5}, \quad
  E_{4} \leftrightarrow E_{6}, \\
%  &h_1 \leftrightarrow h_1+h_4-e_{1}-e_{2}, \quad 
%  e_{1} \leftrightarrow h_4-e_{2}, \\
%  &e_{2} \leftrightarrow h_4-e_{1}, \quad
%  e_{3} \leftrightarrow e_{5}, \quad
%  e_{4} \leftrightarrow e_{6}, \\
\sigma_{12}^{(1)}:\ & 
  \alpha_1^{(1)} \leftrightarrow \alpha_2^{(1)}\quad
  \check{\alpha}_1^{(1)} \leftrightarrow \check{\alpha}_2^{(1)}, \\
  &H_{1} \leftrightarrow H_{4}, \quad
  E_{5} \leftrightarrow E_{7}, \quad
  E_{6} \leftrightarrow E_{8}, \\
%  &h_{1} \leftrightarrow h_{4}, \quad
%  e_{5} \leftrightarrow e_{7}, \quad
%  e_{6} \leftrightarrow e_{8},
%\end{align*}
%\begin{align*}
\sigma_{01}^{(2)}:\ & 
  \alpha_0^{(2)} \leftrightarrow \alpha_1^{(2)}, \quad
  \check{\alpha}_0^{(2)} \leftrightarrow \check{\alpha}_1^{(2)}, \\
  &H_{p_2} \leftrightarrow H_{p_2}+H_{q_2}-E_{9}-E_{10}, \quad 
  E_{9} \leftrightarrow H_{q_2}-E_{10}, \\
  &E_{10} \leftrightarrow H_{q_2}-E_{9}, \quad
  E_{11} \leftrightarrow E_{13}, \quad
  E_{12} \leftrightarrow E_{14}, \\
%  &h_3 \leftrightarrow h_2+h_3-e_{9}-e_{10}, \quad 
% e_{9} \leftrightarrow h_2-e_{10}, \\
%  &e_{10} \leftrightarrow h_2-e_{9}, \quad
%  e_{11} \leftrightarrow e_{13}, \quad
% e_{12} \leftrightarrow e_{14}, \\
\sigma_{12}^{(2)}:\ & 
  \alpha_1^{(2)} \leftrightarrow \alpha_2^{(1)}, \quad
  \check{\alpha}_1^{(2)} \leftrightarrow \check{\alpha}_2^{(2)}, \\
  &H_{2} \leftrightarrow H_{3}, \quad
  E_{13} \leftrightarrow E_{15}, \quad
  E_{14} \leftrightarrow E_{16}, \\
%  &h_{2} \leftrightarrow h_{3}, \quad
%  e_{13} \leftrightarrow e_{15}, \quad
%  e_{14} \leftrightarrow e_{16}, \\
\sigma^{(12)}:\ & 
  \alpha_i^{(1)} \leftrightarrow \alpha_i^{(2)}, \quad
  \check{\alpha}_i^{(1)} \leftrightarrow \check{\alpha}_i^{(2)}, \\
  &H_{1} \leftrightarrow H_{3}, \quad   H_{2} \leftrightarrow H_{4}, \\
  &E_{i} \leftrightarrow E_{i+8} \quad (\mbox{{\rm for} $i=1,2, \dots,8$}).% \\
%  &h_{1} \leftrightarrow h_{3}, \quad   h_{2} \leftrightarrow h_{4}, \\
%  &e_{i} \leftrightarrow e_{i+8} \quad (\mbox{{\rm for} $i=1, 2, \dots, 8$}).
\end{align*}
with the action on $H_2(\mathcal{X}_{{\bf a}}, \Z)$ given by \eqref{H2H2},
where we omit to write for unchanged variables. It is easy to see that each one satisfies Condition (a) and (b) for a Cremona isometry.

\begin{theorem}
The extended affine Weyl group $\widetilde{W}(A_2^{(1)}+A_2^{(1)})$
act on the family of the space of initial conditions $\cX_A$
such that each element $w$ acts as a linear transformation on the set of parameters 
$A=\C^8$ and as a pseudo isomorphisms from $X_{\bf a}$ to $X_{w({\bf a})}$ for generic ${\bf a} \in A$. 
\end{theorem}

\begin{proof}
It is enough to give realisation of the generators as birational mappings on 
$$(q_1, p_2, q_2, p_1;a_0^{(1)},a_1^{(1)},a_2^{(1)},a_0^{(2)},a_1^{(2)},a_2^{(2)};b^{(1)},b^{(2)})\in \C^{14}.$$ 
The following list gives such realisation:
\begin{align*}
w_{\alpha_0^{(1)}}:\ &
  q_1  \leftrightarrow \disp \frac{q_1^2+q_1p_1-b^{(1)} q_1-a_0^{(1)}}{q_1+p_1-b^{(1)}}, \quad
  p_1  \leftrightarrow  \disp \frac{p_1^2+q_1p_1-b^{(1)} p_1+a_0^{(1)}}{q_1+p_1-b^{(1)}}, \\
  &a_0^{(1)} \leftrightarrow -a_0^{(1)}, \quad
  a_1^{(1)} \leftrightarrow a_0^{(1)}+a_1^{(1)}, \quad
  a_2^{(1)} \leftrightarrow a_0^{(1)}+a_2^{(1)}\\
w_{\alpha_1^{(1)}}:\ &
  q_1  \leftrightarrow q_1-a_1^{(1)}p_1^{-1}, \\
  &a_0^{(1)} \leftrightarrow a_0^{(1)}+a_1^{(1)}, \quad
  a_1^{(1)} \leftrightarrow -a_1^{(1)}, \quad
  a_2^{(1)} \leftrightarrow a_1^{(1)}+a_2^{(1)}\\
w_{\alpha_2^{(1)}}:\ &
  p_1  \leftrightarrow  p_1+a_2^{(1)}q_1^{-1}, \\
  &a_0^{(1)} \leftrightarrow a_0^{(1)}+a_2^{(1)}, \quad
  a_1^{(1)} \leftrightarrow a_1^{(1)}+a_2^{(1)}, \quad
  a_2^{(1)} \leftrightarrow -a_2^{(1)}\\
w_{\alpha_0^{(2)}}:\ &
  q_2  \leftrightarrow \disp \frac{q_2^2+p_2q_2-b^{(2)} q_2-a_0^{(2)}}{q_2+p_2-b^{(2)}}, \quad 
  p_2  \leftrightarrow \disp \frac{p_2^2+p_2q_2-b^{(2)} p_2+a_0^{(2)}}{q_2+p_2-b^{(2)}},\\
  &a_0^{(2)} \leftrightarrow -a_0^{(2)}, \quad
  a_1^{(2)} \leftrightarrow a_0^{(2)}+a_1^{(2)}, \quad
  a_2^{(2)} \leftrightarrow a_0^{(2)}+a_2^{(2)}\\
w_{\alpha_1^{(2)}}:\ &
  q_2  \leftrightarrow q_2-a_1^{(2)}p_2^{-1}, \\
  &a_0^{(2)} \leftrightarrow a_0^{(2)}+a_1^{(2)}, \quad
  a_1^{(2)} \leftrightarrow -a_1^{(2)}, \quad
  a_2^{(2)} \leftrightarrow a_1^{(2)}+a_2^{(2)}\\
w_{\alpha_2^{(2)}}:\ &
  p_2  \leftrightarrow p_2+a_2^{(2)}q_2^{-1}, \\
  &a_0^{(2)} \leftrightarrow a_2^{(1)}+a_2^{(2)}, \quad
  a_1^{(2)} \leftrightarrow a_1^{(2)}+a_2^{(2)}, \quad
  a_2^{(2)} \leftrightarrow -a_2^{(2)}
\end{align*}
and
\begin{align*}
\sigma_{01}^{(1)}:\ &
  p_1  \leftrightarrow -q_1-p_1+b^{(1)}, \quad
  a_0^{(1)} \leftrightarrow -a_1^{(1)}, \quad
  a_1^{(1)} \leftrightarrow -a_0^{(1)}, \quad
  a_2^{(1)} \leftrightarrow -a_2^{(1)}\\
\sigma_{12}^{(1)}:\ & 
  q_1  \leftrightarrow p_1, \quad
  a_0^{(1)} \leftrightarrow -a_0^{(1)}, \quad
  a_1^{(1)} \leftrightarrow -a_2^{(1)}, \quad
  a_2^{(1)} \leftrightarrow -a_1^{(1)}\\ 
\sigma_{01}^{(2)}:\ & 
  p_2  \leftrightarrow -q_2-p_2+b^{(2)}, \quad
  a_0^{(2)} \leftrightarrow -a_1^{(2)}, \quad
  a_1^{(2)} \leftrightarrow -a_0^{(2)}, \quad
  a_2^{(2)} \leftrightarrow -a_2^{(2)}\\
\sigma_{12}^{(2)}:\ & 
  q_2  \leftrightarrow p_2, \quad
  a_0^{(2)} \leftrightarrow -a_0^{(2)}, \quad
  a_1^{(2)} \leftrightarrow -a_2^{(2)}, \quad
  a_2^{(2)} \leftrightarrow -a_1^{(2)}\\
\sigma^{(12)}:\ & 
  q_1  \leftrightarrow q_2, \quad p_1  \leftrightarrow p_2, \\
  &a_i^{(1)} \leftrightarrow a_i^{(2)}, \quad  (\mbox{{\rm for} $i=0,1,2$}), \quad
  b^{(1)} \leftrightarrow b^{(2)}.
\end{align*}
 For these computations we used a factorisation formula proposed in \cite{CDT2017} for two-dimensional case, which also works well in the higher dimensional case. 
\end{proof}

The pull-back action $\varphi^*$ on the root lattice is
\begin{align}
(\alpha_0^{(j)}, \alpha_1^{(j)},\alpha_2^{(j)})\mapsto
(\alpha_1^{(j+1)}+\alpha_2^{(j+1)},-\alpha_2^{(j+1)}, \alpha_0^{(j+1)}+ \alpha_2^{(j+1)})
\end{align}
for $j=1,2 \mod 2$, and written by the generators as
\begin{align}
\varphi=\sigma^{(12)}
\circ w_{\alpha_1^{(2)}}\circ \sigma_{12}^{(2)} \circ \sigma_{01}^{(2)} 
\circ w_{\alpha_1^{(1)}}\circ \sigma_{12}^{(1)} \circ \sigma_{01}^{(1)}.
\end{align}
Its action on the variables becomes
\begin{align}\nonumber
&\left(q_1, p_2, q_2, p_1; a_0^{(1)},a_1^{(1)},a_2^{(1)},a_0^{(2)},a_1^{(2)},a_2^{(2)};b^{(1)},b^{(2)}\right)\\
\mapsto&
\Big(-p_2-q_2+b^{(2)}-\frac{a_2^{(2)}}{q_2}, q_1,
-q_1-p_1+b^{(1)}-\frac{a_2^{(1)}}{q_1}, q_2; \\
&\ 
 a_1^{(2)}+a_2^{(2)},-a_2^{(2)}, a_0^{(2)}+a_2^{(2)},
a_1^{(1)}+a_2^{(1)},-a_2^{(1)}, a_0^{(1)}+a_2^{(1)}
;b^{(2)},b^{(1)}\Big),\nonumber
\end{align}
which is the non-autonomous version of $\varphi$. 
The action $(\varphi^2)^*$ on the root lattice is a translation as
\begin{align}
(\alpha_0^{(j)}, \alpha_1^{(j)},\alpha_2^{(j)})\mapsto
(\alpha_0^{(j)}, \alpha_1^{(j)}-\delta^{(j)},\alpha_2^{(j)}+\delta^{(j)})
\end{align}
for $j=1,2$.\\

\noindent Case $A_5^{(1)}$:\\
Let $\cX_A$ denote a family of the space of initial conditions 
$$\cX_A:=\{\cX_{{\bf a}} \mbox{ {\rm in \S4}} ~|~{\bf a}=(a_0,a_1,a_2,a_3,a_4,a_5;b_1,b_2)\in \C^8 \}.$$ 
Let us define root vectors $\alpha_i$ and co-root vectors ($i=0,\dots,5$) as
\begin{align}
\begin{array}{lll}
\alpha_0=H_{q_1}+H_{p_2}-E_{3,4,9,10},
&\alpha_1=H_{q_2}-E_{15,16},
&\alpha_2=H_{p_2}-E_{5,6},\\
\alpha_3=H_{p_1}+H_{q_2}-E_{1,2,11,12},
&\alpha_4=H_{q_1}-E_{7,8},
&\alpha_5=H_{p_1}-E_{13,14}
\end{array}
\end{align}
and
\begin{align}
\begin{array}{lll}
\check{\alpha}_0=h_{p_1}+h_{q_2}-e_{1,2,3,4},
&\check{\alpha}_1=h_{p_2}-e_{15,16},
&\check{\alpha}_2=h_{q_2}-e_{5,6},\\
\check{\alpha}_3=h_{q_1}+h_{p_2}-e_{9,10,11,12},
&\check{\alpha}_4=h_{p_1}-e_{7,8},
&\check{\alpha}_5=h_{q_1}-e_{13,14}.
\end{array}.
\end{align}
Then, the pairing $\langle \alpha_i, \check{\alpha}_j \rangle$ induces 
the affine root system of type $A_5^{(1)}$
with the null vectors $\delta=2H_{q_1,p_1,q_2,p_2}-E_{1,\dots,16}$
and the null co-root vector $\check{\delta}=2h_{q_1,p_1,q_2,p_2} -e_{1,\dots,16}$.
The Cartan matrix and the Dynkin diagram are 
\begin{center}
$\begin{bmatrix}
2&-1&0&0&0&-1\\
-1&2&-1&0&0&0\\
0&-1&2&-1&0&0\\
0&0&-1&2&-1&0\\
0&0&0&-1&2&-1\\
-1&0&0&0&-1&-2
\end{bmatrix}$
\raisebox{-17mm}{
\begin{tikzpicture}[line cap=round,line join=round,>=triangle 45,x=0.9cm,y=0.9cm]
\clip(-2,-2) rectangle (5.4,1.5);
\draw [line width=1.pt] (1.5,0.732)-- (-1.,-1.);
\draw [line width=1.pt] (-1.,-1.)-- (0.25,-1.);
\draw [line width=1.pt] (0.25,-1.)-- (1.5,-1.);
\draw [line width=1.pt] (1.5,-1.)-- (2.75,-1.);
\draw [line width=1.pt] (2.75,-1.)-- (4.0,-1.);
\draw [line width=1.pt] (1.5,0.732)-- (4.0,-1.);
\draw [fill=black] (1.5,0.732) circle (2.5pt);
\draw [fill=black] (-1.,-1.) circle (2.5pt);
\draw [fill=black] (0.25,-1.) circle (2.5pt);
\draw [fill=black] (1.5,-1) circle (2.5pt);
\draw [fill=black] (2.75,-1) circle (2.5pt);
\draw [fill=black] (4.0,-1) circle (2.5pt);
\begin{small}
\draw [color=black] (1.5,1.2) node {$\alpha_0$};
\draw [color=black] (-1.,-1.5) node {$\alpha_1$};
\draw [color=black] (0.25,-1.5) node {$\alpha_2$};
\draw [color=black] (1.5,-1.5) node {$\alpha_3$};
\draw [color=black] (2.75,-1.5) node {$\alpha_4$};
\draw [color=black] (4.,-1.5) node {$\alpha_5$};
\end{small}
\end{tikzpicture}}.
\end{center}
Let $\widetilde{W}(A_5^{(1)})$ denote the extended affine Weyl group $\operatorname{Aut}(A_5^{(1)})\ltimes  W(A_5^{(1)})$.

We define the action of the simple reflection $w_{\alpha_i}$ on the N\'eron-Severi bilattice as \eqref{Rootaction}.
The group of Dynkin automorphisms is generated by 
\begin{align*}
\sigma_{01}: &
  \alpha_0 \leftrightarrow \alpha_1,\quad
  \alpha_2 \leftrightarrow \alpha_5,\quad
  \alpha_3 \leftrightarrow \alpha_4,\quad
  \check{\alpha}_0 \leftrightarrow \check{\alpha}_1,\quad
  \check{\alpha}_2 \leftrightarrow \check{\alpha}_5,\quad
  \check{\alpha}_3 \leftrightarrow \check{\alpha}_4,\\
  &H_{q_2} \leftrightarrow H_{p_2},\quad
  H_{q_1} \leftrightarrow H_{p_1,q_2}-E_{1,2},\quad 
  H_{q_2} \leftrightarrow H_{q_1,p_2}-E_{9,10},\\&
  E_{1} \leftrightarrow H_{p_2}-E_{10},\quad
  E_{2} \leftrightarrow H_{p_2}-E_{9},\quad
  E_{3} \leftrightarrow E_{15},\quad
  E_{4} \leftrightarrow E_{16},\\
  &E_{5} \leftrightarrow E_{13},\quad
  E_{6} \leftrightarrow E_{14},\quad
  E_{7} \leftrightarrow E_{11},\quad
  E_{8} \leftrightarrow E_{12},\\&
  E_{9} \leftrightarrow H_{p_1}-E_{2},\quad
  E_{10} \leftrightarrow H_{p_1}-E_{1},\\
%&h_1 \leftrightarrow h_3,\quad
%  h_2 \leftrightarrow h_3+h_4-e_{1}-e_{2},\quad 
%  h_4 \leftrightarrow h_1+h_2-e_{9}-e_{10},\\&
%  e_{1} \leftrightarrow h_2-e_{10},\quad
%  e_{2} \leftrightarrow h_2-e_{9},\quad
%  e_{3} \leftrightarrow e_{15},\quad
%  e_{4} \leftrightarrow e_{16},\\
%  &e_{5} \leftrightarrow e_{13},\quad
%  e_{6} \leftrightarrow e_{14},\quad
%  e_{7} \leftrightarrow e_{11},\quad
%  e_{8} \leftrightarrow e_{12},\\&
%  e_{9} \leftrightarrow h_4-e_{2},\quad
%  e_{10} \leftrightarrow h_4-e_{1},\\
\sigma_{12}:\ & 
  \alpha_0 \leftrightarrow \alpha_3, \quad
  \alpha_1 \leftrightarrow \alpha_2, \quad
  \alpha_4 \leftrightarrow \alpha_5, \quad
  \check{\alpha}_0 \leftrightarrow \check{\alpha}_3,\quad
  \check{\alpha}_1 \leftrightarrow \check{\alpha}_2,\quad
  \check{\alpha}_4 \leftrightarrow \check{\alpha}_5,\\
  &H_{q_1} \leftrightarrow H_{p_1},\quad 
  H_{q_1} \leftrightarrow H_{p_2},\\&
  E_{1} \leftrightarrow E_{9},\quad
  E_{2} \leftrightarrow E_{10},\quad
  E_{3} \leftrightarrow E_{11},\quad
  E_{4} \leftrightarrow E_{12},\\
  &E_{5} \leftrightarrow E_{15},\quad
  E_{6} \leftrightarrow E_{16},\quad
  E_{7} \leftrightarrow E_{13},\quad
  E_{8} \leftrightarrow E_{14},
%\\%&
%  h_1 \leftrightarrow h_4,\quad
%  h_2 \leftrightarrow h_3,\\&
%  e_{1} \leftrightarrow e_{9},\quad
%  e_{2} \leftrightarrow e_{10},\quad
%  e_{3} \leftrightarrow e_{11},\quad
%  e_{4} \leftrightarrow e_{12},\\
%  &e_{5} \leftrightarrow e_{15},\quad
%  e_{6} \leftrightarrow e_{16},\quad
%  e_{7} \leftrightarrow e_{13},\quad
%  e_{8} \leftrightarrow e_{14}.
\end{align*}
with the action on $H_2(\mathcal{X}_{{\bf a}}, \Z)$ given by \eqref{H2H2}.

\begin{theorem}
The extended affine Weyl group $\widetilde{W}(A_5^{(1)})$
act on the family of the space of initial conditions $\cX_A$
such that each element $w$ acts as a linear transformation on the set of parameters 
$A=\C^8$ and as a pseudo-isomorphisms from $X_{\bf a}$ to $X_{w({\bf a})}$ for generic ${\bf a} \in A$. 
\end{theorem}

\begin{proof}
The following list gives realisation of the generators as birational mappings on 
$$(q_1, p_1,q_2, p_2;a_0,a_0,a_1,a_2,a_3,a_4,a_5;b_1,b_2)\in \C^{14}.$$
\begin{align*}
w_{\alpha_0}:\ &
  p_1  \leftrightarrow  \disp \frac{(q_1+p_2-b_1)p_1-a_0}{q_1+p_2-b_1},\quad
  q_2  \leftrightarrow \disp \frac{(q_1+p_2-b_1)q_2+a_0}{q_1+p_2-b_1},\\
  &a_5 \leftrightarrow a_0+a_5\quad
  a_0 \leftrightarrow -a_0,\quad
  a_1 \leftrightarrow a_0+a_1,\\
w_{\alpha_1}:\ &
  p_2  \leftrightarrow p_2-a_1q_2^{-1},\\
  &a_0 \leftrightarrow a_0+a_1,\quad
  a_1 \leftrightarrow -a_1,\quad
  a_2 \leftrightarrow a_1+a_2\\
w_{\alpha_2}:\ &
  q_2  \leftrightarrow q_2+a_2p_2^{-1},\\
  &
  a_1 \leftrightarrow a_1+a_2,\quad
  a_2 \leftrightarrow -a_2\quad
  a_3 \leftrightarrow a_2+a_3,\\
w_{\alpha_3}:\ &
  q_1  \leftrightarrow \disp \frac{(q_2+p_1-b_2)q_1+a_3}{q_2+p_1-b_2},\quad
  p_2  \leftrightarrow  \disp \frac{(q_2+p_1-b_2)p_2-a_3}{q_2+p_1-b_2},\\
  &a_2 \leftrightarrow a_2+a_3\quad
  a_3 \leftrightarrow -a_3,\quad
  a_4 \leftrightarrow a_3+a_4,\\  
w_{\alpha_4}:\ &
  p_1  \leftrightarrow p_1-a_4 q_1^{-1},\\
  &a_3 \leftrightarrow a_3+a_4,\quad
  a_4 \leftrightarrow -a_4,\quad
  a_5 \leftrightarrow a_4+a_5,\\
w_{\alpha_5}:\ &
  q_1  \leftrightarrow q_1+a_5 p_1^{-1},\\
  &
  a_4 \leftrightarrow a_4+a_5,\quad
  a_5 \leftrightarrow -a_5\quad
  a_0 \leftrightarrow a_0+a_5,\\
%\end{align*}
%\begin{align*}
\sigma_{01}:\ &
  q_1  \leftrightarrow -q_2-p_1+b_2,\quad
  p_1  \leftrightarrow p_2,\quad  
  q_2  \leftrightarrow -q_1-p_2+b_1, \\&
  a_0 \leftrightarrow -a_1,\quad
  a_2 \leftrightarrow -a_5,\quad
  a_3 \leftrightarrow -a_4,\quad
  b_1 \leftrightarrow b_2,\\
\sigma_{12}:\ & 
  q_1  \leftrightarrow p_1,\quad p_2  \leftrightarrow q_2,\\&
  a_0 \leftrightarrow -a_3,\quad
  a_1 \leftrightarrow -a_2,\quad
  a_4 \leftrightarrow -a_5,\quad
  b_1 \leftrightarrow b_2
\end{align*}
\end{proof}

The pull-back action of $\varphi^*$ on the root lattice is
\begin{align}
(\alpha_0,\dots,\alpha_5)\mapsto
(\alpha_1+\alpha_2, \alpha_3+\alpha_4,-\alpha_4,\alpha_4+\alpha_5,\alpha_0+\alpha_1, -\alpha_1)
\end{align}
and written by the generators as
\begin{align}
\varphi= w_{\alpha_5}\circ w_{\alpha_2}\circ \sigma_{01} \circ \sigma_{12}.
\end{align}
Its action on the variables becomes
\begin{align}\nonumber
&\left(q_1,p_1,q_2,p_2; a_0,a_1,a_2,a_3,a_4,a_5;b_1,b_2\right)\\
\mapsto&
\Big(-q_1-p_2+b_1+\frac{a_1}{q_2},q_2,
-q_2-p_1+b_2+\frac{a_4}{q_1},q_1; \\
&\quad\ 
a_1+a_2,a_3+a_4,-a_4,a_4+a_5,a_0+a_1,-a_1; b_1,b_2 \Big),\nonumber
\end{align}
which is the non-autonomous version of $\varphi$. 
It is easy to see that $(\varphi^4)^*$ is a translation on the root lattice
as
\begin{align*}
&\varphi^4:(\alpha_0,\dots,\alpha_5)\mapsto
(\alpha_0,\dots,\alpha_5)+\delta (0,1,-1,0,1,-1).
\end{align*}

\begin{remark}\label{twin}
It is highly nontrivial to find the root basis. For example, since the difference of decomposition of the anti-canonical divisor between 
the $A_2^{(1)}+A_2^{(1)}$ case and the $A_5^{(1)}$ case is just
exchange of $H_{q_1}$ and $H_{q_2}$, 
for the $A_2^{(1)}+A_2^{(1)}$ variety, 
the $A_5^{(1)}$ root system with the basis:
\begin{align*}
&\begin{array}{lll}
\alpha_0=H_{q_2}+H_{p_2}-E_{3,4,9,10},
&\alpha_1=H_{q_1}-E_{15,16},
&\alpha_2=H_{p_2}-E_{5,6},\\
\alpha_3=H_{p_1}+H_{q_1}-E_{1,2,11,12},
&\alpha_4=H_{q_2}-E_{7,8},
&\alpha_5=H_{p_2}-E_{13,14}
\end{array}\\
&\begin{array}{lll}
\check{\alpha}_0=h_{p_1}+h_{q_1}-e_{1,2,3,4},
&\check{\alpha}_1=h_{p_2}-e_{15,16},
&\check{\alpha}_2=h_{q_1}-e_{5,6},\\
\check{\alpha}_3=h_{q_2}+h_{p_2}-e_{9,10,11,12},
&\check{\alpha}_4=h_{p_1}-e_{7,8},
&\check{\alpha}_5=h_{q_2}-e_{13,14}.
\end{array}.
\end{align*}
also satisfies Condition (a) and (b) for Cremona isometries. However, it does not satisfy Condition (c).
Actually, $w_{\alpha_1}$ acts to an effective divisor $E_{16}$ as $E_{16}\mapsto H_{q_1}-E_{15}$, but
$H_{q_1}-E_{15}$ is not effective. Similarly, for the $A_5^{(1)}$
variety, the $A_2^{(1)}+A_2^{(1)}$ root system with the basis:
\begin{align*}
&\begin{array}{lll}
\alpha_0^{(1)}=H_{q_2}+H_{p_1}-E_{1,2,3,4}, 
&\alpha_1^{(1)}=H_{p_1}-E_{5,6}, 
&\alpha_2^{(1)}=H_{q_1}-E_{7,8}, \\
\alpha_0^{(2)}=H_{p_2}+H_{q_1}-E_{9,10,11,12},
&\alpha_1^{(2)}=H_{p_2}-E_{13,14}, 
&\alpha_2^{(2)}=H_{q_2}-E_{15,16}
\end{array}\\
&\begin{array}{lll}
\check{\alpha}_0^{(1)}=h_{q_2}+h_{p_1}-e_{1,2,3,4}, 
&\check{\alpha}_1^{(1)}=h_{q_2}-e_{5,6},
&\check{\alpha}_2^{(1)}=h_{p_1}-e_{7,8},\\
\check{\alpha}_0^{(2)}=h_{q_1}+h_{p_2}-e_{9,10,11,12},
&\check{\alpha}_1^{(2)}=h_{q_1}-e_{13,14},
&\check{\alpha}_2^{(2)}=h_{p_2}-e_{15,16}
\end{array}.
\end{align*}
also satisfies Condition (a) and (b), but does not satisfy (c).
\end{remark}

\section{Continuous flow}

As commented in Introduction, the conserved quantities $I_1$ and $I_2$ in \eqref{I12} or \eqref{I12.2} give Hamiltonian flows commuting with each other.
In this section we consider its non-autonomous version. That is, we consider Hamiltonian system of the form
\begin{align}\label{NAH} 
 \frac{dq_1}{dt}=\frac{\partial I}{\partial p_1}, &\quad  \frac{dp_1}{dt}=- \frac{\partial I}{\partial q_1},& 
 \frac{dq_2}{dt}=\frac{\partial I}{\partial p_2}, &\quad  \frac{dp_2}{dt}=- \frac{\partial I}{\partial q_2}
\end{align}
which is regular on a family of surfaces $\cX_{\bf a} \setminus {\mathcal D}$, where ${\mathcal D}$ is the support set of the singular anti-canonical divisor $\sum_i m_i D_i$.  

To find non-autonomous Hamiltonian, we use a technique used by Takano and his collaborators in \cite{ST1997, MMT1999} and by Sasano-Yamada in \cite{SY2007}. We start from general polynomial $I(q_1,p_1,q_2,p_2)$ of order $(2,2,2,2)$ and assume the Hamilton system \eqref{NAH} to be holomorphic on $\cX_{\bf a} \setminus {\mathcal D}$. \\

\begin{theorem}
Case $A_2^{(1)}+A_2^{(1)}$: 
Let $I_1^{{\rm NA}}$ and $I_2^{{\rm NA}}$ be defined as
\begin{align*}
I_1^{{\rm NA}}=&H_{{\rm IV}}(-q_1, p_1;a_1^{(1)},a_2^{(1)}; b^{(1)})= q_1p_1(q_1+p_1-b^{(1)})-a_1^{(1)} q_1+a_2^{(1)}p_1\\
I_2^{{\rm NA}}=&H_{{\rm IV}}(-q_2, p_2;a_1^{(2)},a_2^{(2)}; b^{(2)})=
 p_2q_2(p_2+q_2-b^{(2)})-a_1^{(2)} q_2+a_2^{(2)}p_2,
\end{align*}
where
\begin{align*}
H_{{\rm IV}}(q, p;\alpha, \beta;t)=& qp(p-q-t)+\alpha q+\beta p
\end{align*}
is the Hamiltonian of the forth Painlv\'e equation \cite{Okamoto1980}.
Then the Hamiltonian system \eqref{NAH} with $I=I_i^{{\rm NA}}$, i=1,2, and
\begin{align}
\frac{db^{(j)}}{dt}=&\left\{ \begin{array}{ll} \lambda^{(j)}:=a_0^{(j)}+a_1^{(j)}+a_2^{(j)} \quad& (\mbox{{\rm if} $i=j$})\\[2mm]
0&(\mbox{{\rm if} $i\neq j$})\end{array}\right.
\end{align}
is regular on $\cX_{\bf a} \setminus {\mathcal D}$.\\

\noindent Case $A_5^{(1)}$: 
Let $I_2^{{\rm NA}}$ be defined as
\begin{align*}
&H_{{\rm V}}(q_1,p_1;a_5,a_4,-\frac{1}{2}(a_0+a_2+a_4-a_1-a_3-a_5); b_1,-b_2)\\
&+H_{{\rm V}}(q_2,p_2;a_2,a_1,\frac{1}{2}(a_0+a_2+a_4-a_1-a_3-a_5); b_2,-b_1)-2q_1p_1q_2p_2,
\end{align*}
where
\begin{align*}
H_{{\rm V}}(q,p;\alpha,\beta,\gamma;s,t)=& pq(p+t)(q-s)+\alpha tq+\beta sp+\gamma pq
\end{align*}
is the Hamiltonian of the fifth Painlv\'e equation \cite{Okamoto1980}.
Then the Hamiltonian system \eqref{NAH} with $I=I_2^{{\rm NA}}$ and
\begin{align}
\frac{db_j}{dt}=& \frac{1}{2}\lambda b_j =\frac{1}{2}(a_0+a_1+\cdots+a_5)b_j, \quad (j=1,2)\label{bd_evo}
\end{align}
is regular on $\cX_{\bf a} \setminus {\mathcal D}$.
\end{theorem}

\begin{proof}
The proof is straightforward but long and we omit the details, but
for example, in the case $A_5^{(1)}$, on the coordinates $U_4$ in Theorem~\ref{SIC},
the ODE becomes regular i.e.
\begin{align*}
\frac{dq_1}{dt}=&\frac{f_1(q_1,v_2,u_4,v_4)}{2 (-1 + v_2 u_4)},
&\frac{dv_2}{dt}=&\frac{f_2(q_1,v_2,u_4,v_4)}{2 (-1 + v_2 u_4)},\\
\frac{du_4}{dt}=&\frac{f_3(q_1,v_2,u_4,v_4)}{2 (-1 + v_2 u_4)},
&\frac{dv_4}{dt}=&\frac{f_4(q_1,v_2,u_4,v_4)}{2 (-1 + v_2 u_4)},
\end{align*}
where $f_i(q_1,v_2,u_4,v_4)$'s are polynomials of $q_1,v_2,u_4,v_4$.
Note that, since $b_1$ depends on $t$ by \eqref{bd_evo}, $dv_3/dt$ is computed as
$$\frac{dv_3}{dt}= \frac{1}{u_1}\left(\frac{dq_1}{dt}  + \frac{dp_2}{dt} - \frac{db_1}{dt}\right)
 - \frac{q_1 + p_2 - b_1}{u_2^2} \frac{du_2}{dt}.$$

\end{proof}

\begin{remark}
\begin{enumerate}
\item If we set the parameters $a_i$'s as original autonomous one, $I_i^{{\rm NA}}$ recovers $I_i$.
\item In Case $A_5^{(1)}$, $I_2^{{\rm NA}}$ is the unique polynomial of degree $(2,2,2,2)$
except for the constant term which gives a regular Hamiltonian flow with \eqref{bd_evo}. 
\item
The Hamilton system with $I=I_2^{{\rm NA}}$ for Case $A_5^{(1)}$ is slightly modified version of the $A_5^{(1)}$ member in Noumi-Yamada's higher order Pailev\'e equations of type $A_l^{(1)}$ (original one is proposed in \cite{NY1998} and modified version is  proposed in \cite{SY2007}. The latter also gives a discreption as a Hamiltonian system of coupled $P_{{\rm V}}$ equations). Indeed, setting
$f_0=b_1 - q_1 - p_2$, $f_1 = q_2$, $f_2 = p_2$, $f_3=b_2 - q_2 - p_1$, $f_4 = q_1$ and $f_5 = p_1$, we have 
\begin{align*}
\frac{df_i}{dt}=&f_i(-f_{i+1}f_{i+2}-f_{i+1}f_{i+4}-f_{i+3}f_{i+4}+f_{i+2}f_{i+3}+f_{i+2}f_{i+5}+f_{i+4}f_{i+5})\\
&-\frac{1}{2}(a_{i}+a_{i+2}+a_{i+4}-a_{i+1}-a_{i+3}-a_{i+5})f_i+a_i(f_i+f_{i+2}+f_{i+4})
\end{align*}
for $i=0,1,\dots,5$,
where indices $0,1,\dots,5$ are regarded as elements of $\Z/6\Z$.  
\end{enumerate}
\end{remark}

\section{Concluding remarks}

In this paper we investigated two integrable 4-dimensional mappings and constructed
the space of initial conditions on the level of pseudo-auto/isomorphisms. 

The deautonomised version of the first mapping is turned out to be a B\"acklund transformation for the direct product of the fourth Painlev\'e equation with itself
and the symmetry group is $A_2^{(1)}+A_2^{(1)}$ affine Weyl group. This situation is easily generalised to $X_l^{(1)}+X_m^{(1)}$ affine Weyl group, where $X_l^{(1)}$ and $X_m^{(1)}$
are affine Weyl subgroup in $E_8^{(1)}$ appearing in Sakai's classification of two-dimensional discrete Painlev\'e equations, i.e. $X=A, D, E$ ,$l,m=0,1,2,3,4,5,6,7,8$.
In this case the variety is almost (except intersection points of centers of blowups as Remark~\ref{int_centers}) the direct product of Sakai studied generalised Halphen surfaces \cite{Sakai2001}. Here, it is allowed that additive, multiplicative and elliptic difference systems are mixed but independently for $2+2$ variables.  

The second mapping was obtained just by switching two terms in the first mapping,
but this simple surgery generates a variety with a different type symmetry.   
On the level of cohomology, the only difference is the decompositions of the anti-canonical divisors as \eqref{KXdecomposition} and \eqref{KXdecomposition2}.
Moreover, as commented in Remark~\ref{twin}, their symmetries are closely related with each other. We expect that there are many such ``twin'' phenomena.

%%%%%%%%%%%%%%%%%%%%%%%%%%%%%
\subsection*{Acknowledgement}
%%%%%%%%%%%%%%%%%%%%%%%%%%%%%
T.~T. was supported by the Japan Society for the
Promotion of Science, Grand-in-Aid (C) (17K05271).

\end{document}